\documentclass{amsart}

\usepackage{amsmath}
\usepackage{amsfonts}
\usepackage{mathtools}
\usepackage{amssymb,amsthm,mathrsfs}
\usepackage[utf8]{inputenc}
\usepackage{hyperref}
\usepackage{enumitem}
\usepackage{microtype}
\usepackage{needspace}
\usepackage{xcolor}

\hypersetup{
  colorlinks=true,
  linkcolor=blue!45!black,
  citecolor=green!35!black,
  urlcolor=blue!45!black,
  pdftitle={The Bernstein--Sato polynomial of the Hankel determinant},
  pdfauthor={Xinyu Lin},
  pdfsubject={Bernstein--Sato polynomial of the symmetric Hankel determinant},
  pdfkeywords={Bernstein--Sato polynomial, Hankel determinant, nearby cycles,
  Archimedean zeta function, Selberg integral, secant variety}
}

\allowdisplaybreaks
\numberwithin{equation}{section}
\theoremstyle{plain}
\newtheorem{theorem}{Theorem}
\numberwithin{theorem}{section}
\newtheorem{prop}[theorem]{Proposition}
\newtheorem{lemma}[theorem]{Lemma}
\newtheorem{coro}[theorem]{Corollary}
\theoremstyle{definition}

\newtheorem{example}[theorem]{Example}
\newtheorem*{aistatement}{AI statement}
\newtheorem*{ack}{Acknowledgment}
\theoremstyle{remark}

\newcommand{\A}{\mathbb{A}}
\newcommand{\C}{\mathbb{C}}
\renewcommand{\P}{\mathbb{P}}
\newcommand{\Q}{\mathbb{Q}}
\newcommand{\R}{\mathbb{R}}
\newcommand{\Z}{\mathbb{Z}}

\newcommand{\cA}{\mathcal{A}}
\newcommand{\cM}{\mathcal{M}}
\newcommand{\cQ}{\mathcal{Q}}
\newcommand{\cR}{\mathcal{R}}
\newcommand{\cZ}{\mathcal{Z}}
\newcommand{\cO}{\mathscr{O}}
\newcommand{\shD}{\mathscr{D}}

\newcommand{\into}{\hookrightarrow}
\newcommand{\wt}{\widetilde}
\newcommand{\wb}{\overline}
\newcommand{\IC}{\operatorname{IC}}
\newcommand{\Supp}{\operatorname{Supp}}
\newcommand{\ord}{\operatorname{ord}}
\newcommand{\e}{\mathrm e}
\newcommand{\ii}{\mathrm i}
\newcommand{\dd}{\,\mathrm d}

\title[The Bernstein--Sato polynomial of the Hankel determinant]
{The Bernstein--Sato polynomial of the Hankel determinant}
\author{Xinyu Lin}
\address{Xinyu Lin, School of Mathematical Sciences, Zhejiang University, Hangzhou 310058, P. R. China}
\email{xinyulin@zju.edu.cn}


\begin{document}

\begin{abstract}
Let
\[
f_m(x_0,\ldots,x_{2m-2})
=\det(x_{i+j})_{0\leq i,j\leq m-1}
\]
be the determinant of the generic symmetric Hankel matrix of size $m$.
We determine its global Bernstein--Sato polynomial.  The proof combines the
nearby cycles of the corresponding secant hypersurface with an explicit
Archimedean zeta integral over the cone of positive definite Hankel matrices.
A Gauss--Radau parametrization reduces this integral to a Morris integral and
hence to a quotient of Gamma functions.  Its primitive poles are detected by
test functions supported arbitrarily close to the vertex of the cone.
Lichtin's theorem, applied to a secant-stratum log resolution, fixes the
integer shift between these poles and the roots of the Bernstein--Sato
polynomial.  A relative $\shD[s]$-lattice realization of nearby cycles,
together with the simplicity of the corresponding intersection complexes,
gives uniqueness and multiplicity one in every monodromy class.  The induction
in the matrix size follows from the local product structure of secant
varieties.
\end{abstract}

\maketitle
\hypersetup{linkcolor=black}
\tableofcontents
\hypersetup{linkcolor=blue!45!black}

\section{Introduction}

Let $f\in\C[x_1,\ldots,x_N]$ be a nonconstant polynomial, and let
\[
D_N=\C\langle x_1,\ldots,x_N,
\partial_{x_1},\ldots,\partial_{x_N}\rangle
\]
be the $N$-th Weyl algebra of polynomial differential operators.  The
Bernstein--Sato polynomial is defined by a functional equation
\begin{equation}\label{eq:bernstein-equation-intro}
 P(s)f^{s+1}=b(s)f^s,
 \qquad P(s)\in D_N[s],\quad b(s)\in\C[s].
\end{equation}
Here $s$ is a central indeterminate and $f^s$ is a formal symbol; the precise
$\shD[s]$-module interpretation is recalled in
Section~\ref{subsec:dmodules}.  Bernstein proved that the polynomials $b(s)$
occurring in \eqref{eq:bernstein-equation-intro} form a nonzero ideal of
$\C[s]$ \cite{Bernstein}.  Its monic generator is the
\emph{Bernstein--Sato polynomial}, or $b$-function, of $f$ and is denoted by
$b_f(s)$.  Replacing polynomial differential operators by their germs at a
point $x$ gives the local Bernstein--Sato polynomial $b_{f,x}(s)$.

The polynomial $b_f(s)$ is a subtle invariant of the hypersurface $(f=0)$.
Its roots are negative rational numbers \cite{Kashiwara}, and if $\xi$ is a
local root, then $\exp(-2\pi\ii\xi)$ occurs as an eigenvalue of nearby-cycle
monodromy \cite{Kashiwara,Malgrange}.  Iterating
\eqref{eq:bernstein-equation-intro} also gives meromorphic continuation and
restricts the possible poles of Archimedean zeta distributions
\cite{Bernstein}.  These structural theorems do not by themselves determine
the exact polynomial.  Explicit formulas are most accessible in the presence
of substantial symmetry.  For example, for the determinant of a generic
$d\times d$ matrix one has
\[
 b_{\det}(s)=\prod_{i=1}^{d}(s+i)
\]
and, more generally, equivariant and representation-theoretic methods apply
to prehomogeneous relative invariants, quiver semi-invariants, maximal
minors, and submaximal Pfaffians; see, for example,
\cite{SatoShintani,LorinczSlices,LRWW}.

The polynomial considered here is determinantal but substantially more
structured.  For $m\geq 1$, let
\begin{equation}\label{eq:def-fm}
H_m(x)=
\begin{pmatrix}
x_0&x_1&x_2&\cdots&x_{m-1}\\
x_1&x_2&x_3&\cdots&x_m\\
x_2&x_3&x_4&\cdots&x_{m+1}\\
\vdots&\vdots&\vdots&\ddots&\vdots\\
x_{m-1}&x_m&x_{m+1}&\cdots&x_{2m-2}
\end{pmatrix},
\qquad
f_m(x)=\det H_m(x),
\end{equation}
where $x=(x_0,\ldots,x_{2m-2})$.  The defining Hankel condition is
$H_{i,j}=H_{i',j'}$ whenever $i+j=i'+j'$.  Thus $f_m$ is homogeneous of
degree $m$ in
\begin{equation}\label{eq:N}
N=N_m=2m-1
\end{equation}
variables.  We write $b_{f_m}(s)$ for its global Bernstein--Sato
polynomial.  The rational normal curve and its secant varieties carry a
natural $SL_2(\C)$-action, under which $f_m$ is invariant.  The $b$-function
theory for relative invariants of prehomogeneous vector spaces is developed
in \cite[Sections 3--4]{SatoPVS}.  For
$m\geq3$, however, $SL_2(\C)\times\C^\times$ has no open orbit on
$\C^{2m-1}$, so this method does not apply.  We instead use the secant
stratification and its local product structure.

Our main result is as follows.

\begin{theorem}\label{thm:main-intro}
For every $m\geq 1$,
\begin{equation}\label{eq:main-formula}
b_{f_m}(s)
=(s+1)\prod_{q=2}^{m}
\left[
\left(s+2-\frac1q\right)
\prod_{\substack{1\leq k\leq q-2\\(k,q)=1}}
\left(s+2+\frac{k}{q}\right)
\right].
\end{equation}
All roots in \eqref{eq:main-formula} are simple and pairwise distinct.  In
particular,
\[
\deg b_{f_m}=1+\sum_{q=2}^{m}\varphi(q).
\]
Here $\varphi$ denotes Euler's totient function.
\end{theorem}

The hypersurface $D_m=(f_m=0)\subset\C^{2m-1}$ is the affine cone over the
largest proper secant variety of the rational normal curve of degree $2m-2$.
Brogan proved that the nearby cycles of $f_m$ are intersection complexes of
rank-one local systems supported on the secant strata \cite{Brogan}.  These
results determine all possible monodromy eigenvalues and hence the possible
root classes modulo $\Z$.  They also show that every nearby-cycle
eigensummand is simple.  They do not, however, determine which integral
translate occurs in a given class or the multiplicity of that root.

The multiplicity and uniqueness questions are settled by the
Beilinson--Bernstein construction in the form proved by Wu \cite{Wu}.  After
localizing $\C[s]$ at a fixed monodromy class, the adjacent quotients
\[
\frac{\shD[s]f^{s-k}}{\shD[s]f^{s-k+1}}
\]
form the successive factors of a filtration.  The factor indexed by $k$ is
nonzero exactly when the corresponding integral translate is a root.
Simplicity allows only one nonzero factor.  Moreover, the central endomorphism
induced by $s$ is nilpotent; simplicity forces this nilpotent endomorphism to
vanish, which gives multiplicity one.

It remains to prove existence and to identify the translate.  The new classes
at level $m$ are the monodromy classes of exact order $m$; Brogan's support
calculation shows that they occur only at the vertex.  To detect them, we use
the positive definite Hankel cone
\[
\Omega_m=\{x\in\R^{2m-1}:H_m(x)>0\}.
\]
Set
\[
\ell_m(x)=\sum_{j=0}^{m-1}\binom{m-1}{j}x_{2j}.
\]
A node-and-weight construction, equivalent to a Gauss--Radau formula for the
truncated moment functional, gives global coordinates on a full-measure open
subset of $\Omega_m$.  In these coordinates $f_m$, the Jacobian, and
$\ell_m$ all factor.  The resulting Laplace zeta integral
\begin{equation}\label{eq:intro-zeta}
\cZ_m(s)=\int_{\Omega_m}
\e^{-\ell_m(x)}f_m(x)^s\dd x
\end{equation}
reduces to the Morris integral and admits the exact evaluation
\begin{equation}\label{eq:intro-gamma}
\cZ_m(s)=
2^{m-1-m(m-1)(s+2)}(2\pi)^{m-1}\Gamma(s+1)
\prod_{j=2}^{m}
\frac{\Gamma(j(s+2)-1)}{\Gamma((j-1)(s+2))}.
\end{equation}
The final numerator in \eqref{eq:intro-gamma} is
$\Gamma(m(s+2)-1)=\Gamma(ms+N)$, precisely the radial factor produced by
blowing up the vertex.  At
\[
s=-2+\frac1m,
\qquad
s=-2-\frac{k}{m}
\quad(1\leq k\leq m-2,\ (k,m)=1),
\]
this factor has a simple pole.  The remaining Gamma quotient is finite and
nonzero there, so integration over the exceptional directions cannot cancel
these poles.

A pole at $s_0$ implies only that $b_f(s_0+j)=0$ for some
$j\in\Z_{\geq0}$; in general it does not imply that $b_f(s_0)=0$.  To remove
this integer ambiguity, we use the affine log resolution obtained by first
blowing up the vertex of the cone and then pulling back Bertram's projective
secant resolution.  The resulting divisors $E_q$, $1\leq q\leq m$, have
numerical data
\[
\ord_{E_q}(f_m)=q,
\qquad
\ord_{E_q}(K_{\wt Y_m/\C^N})=2q-2.
\]
Lichtin's theorem, as stated in
\cite[Theorem 3.5.1]{PopaNotes}, then says that every root is of the form
\[
-\frac{2q-1+\ell}{q},\qquad \ell\in\Z_{\geq0}.
\]
If the monodromy eigenvalue has exact order $m$, the reduced denominator forces
$q=m$, and the corresponding exponent is at least
$N/m=2-1/m$.  This lower bound forces $j=0$ at each of the poles displayed
above.  Finally, Brogan's local Hankel elimination identifies the transverse
singularity along a lower secant stratum with a smaller determinant $f_q$.
Consequently $b_{f_q}(s)$ divides $b_{f_m}(s)$ for $q<m$.  The vertex supplies
the new primitive order-$m$ roots, while the transverse slices supply all
roots inherited from smaller sizes; this closes the induction.

We use Brogan's secant-chart and nearby-cycle theorems
\cite[Lemma 3.1, Corollary 3.2, Theorems 3.17 and 3.22]{Brogan}, Bertram's
secant resolution \cite{Bertram} in the detailed form given in
\cite[Section 3]{SchnellYang}, and Lichtin's resolution-theoretic theorem
\cite[Theorem 5]{Lichtin}, as formulated in
\cite[Theorem 3.5.1]{PopaNotes}.  On the $D$-module side, we use Wu's
$b$-function construction and Riemann--Hilbert comparison for nearby cycles
\cite[Theorem 1.1, Theorem 2.10, Definition 2.11, and Proposition 2.12]{Wu}.
For meromorphic continuation we use Bernstein's componentwise theorem
\cite{Bernstein}, while the explicit analytic calculation rests on the Morris
integral evaluation \cite[(1.17)--(1.18)]{ForresterWarnaar}.  The hypotheses of
these results are checked where they are applied.

The paper is organized as follows.  Section~\ref{sec:geometry} first recalls
the $D$-module and nearby-cycle constructions used in the proof and then
collects the secant geometry, combining the two inputs to control integral
translates and multiplicities of roots.  The positive Hankel cone and its
Gauss--Radau coordinates are established in Section~\ref{sec:cone}.  The zeta
integral is evaluated and its primitive poles are isolated in
Section~\ref{sec:zeta}.  In Section~\ref{sec:poles-roots} we pass from those
poles to roots at their exact positions.  The induction and the proof of
Theorem~\ref{thm:main-intro} are completed in Section~\ref{sec:induction}.  The
coordinate construction is illustrated for a Hankel matrix of size three in
Example~\ref{ex:m3-coordinates}; factorizations of
the resulting low-rank $b$-functions and reproducible Macaulay2 commands appear
after the proof, in Section~\ref{sec:checks}.

\begin{aistatement}
This work was partially assisted by GPT-5.6-Sol.  Specifically, the
model proposed the Morris-integral strategy used in the proof of
Theorem~\ref{thm:zeta-formula} and supplied preliminary arguments for the
node-and-weight construction in Theorem~\ref{thm:radau}.  The author
subsequently refined and finalized these arguments.
\end{aistatement}

\begin{ack}
The author thanks his advisor, Lei Wu, for suggesting this problem, for many
helpful discussions, and for valuable comments on an earlier version of the
manuscript.
\end{ack}

\section{\texorpdfstring{$D$-modules}{D-modules} and secant-geometric
preliminaries}\label{sec:geometry}

\subsection{Differential operators, $b$-functions, and nearby cycles}
\label{subsec:dmodules}

We recall the $D$-module setup used later; see
\cite[Sections 2.1--2.2]{Wu}.

Let $Y$ be a smooth complex algebraic variety, let $f$ be a nonzero regular
function, put $D=(f=0)$ and $U=Y\setminus D$, and denote the inclusions by
\[
 j:U\into Y,
 \qquad
 i:D\into Y.
\]

For a holonomic $\shD_U$-module, we write $j_*$ for its maximal
direct-image extension across $D$, and $j_!$ for the minimal extension
defined by holonomic duality.  For the closed embedding $i$, the notation
$i_*$ denotes the direct-image functor.  We also write
$\operatorname{DR}$ for the de Rham functor.

The sheaf $\shD_Y$ is the sheaf of algebraic differential operators on $Y$,
and
\[
\shD_Y[s]=\shD_Y\otimes_{\C}\C[s]
\]
is obtained by adjoining a central indeterminate $s$.  Here $\cO_Y$ denotes
the structure sheaf of regular functions, viewed as a left $\shD_Y$-module
through differentiation.  Recall that a coherent $\shD_Y$-module is called
holonomic when its characteristic variety has the smallest possible
dimension, namely $\dim Y$.  Regular holonomic $\shD_Y$-modules are the
algebraic differential equations with regular singularities; under the
regular Riemann--Hilbert correspondence, they correspond to perverse sheaves.
A holonomic module is \emph{simple} if it has no nonzero proper
$\shD_Y$-submodule.  These are the only structural properties of holonomic
$D$-modules used below; see \cite[Chapters 3 and 7]{HTT} for the underlying
theory.

For $Y=\C^N$, the global sections of $\shD_{\C^N}$ form the Weyl algebra
$D_N$ generated by the coordinate functions $x_i$ and the derivations
$\partial_{x_i}$.  In this case
\[
D_N[s]=D_N\otimes_{\C}\C[s].
\]
For sheaves of differential operators, when the underlying space is clear we
suppress the subscript and write simply $\shD[s]$.

Let $\cM$ be a left holonomic $\shD_Y$-module, put $\cM_U=\cM|_U$, and
choose a coherent $\cO_Y$-submodule
$\cM_0\subset j_*\cM_U$ such that
$\cM_U=\shD_U\cdot\cM_0|_U$, as in \cite[Section 2.1]{Wu}.  The formal
symbol $f^s$ gives the free
$\C[s]$-module
\[
 j_*(\cM_U[s]\cdot f^s)=j_*(\cM_U)\cdot f^s\otimes_{\C}\C[s].
\]
Thus $\cM_U[s]=\cM_U\otimes_{\C}\C[s]$; the factor $f^s$ is a symbol, not
a single-valued function on $U$.
Its $\shD_Y[s]$-action is characterized, for a vector field $v$, by
\begin{equation}\label{eq:formal-power-action}
 v(mf^s)=\left(v(m)+s\frac{v(f)}{f}m\right)f^s.
\end{equation}
For $k\in\Z$, set
\[
 \shD_Y[s]\cM_0\cdot f^{s+k}
 \subset j_*(\cM_U[s]\cdot f^s)
\]
for the coherent submodule generated by $\cM_0f^{s+k}$.  The $b$-function of
$\cM_U$ along $f$, relative to the chosen lattice $\cM_0$, is the monic
polynomial of least degree that annihilates
\begin{equation}\label{eq:wus-b-quotient}
 \frac{\shD_Y[s]\cM_0\cdot f^s}
      {\shD_Y[s]\cM_0\cdot f^{s+1}}.
\end{equation}
When $\cM_U=\cO_U$ and $\cM_0=\cO_Y$, this is the usual Bernstein--Sato
polynomial $b_f(s)$.

For $\alpha\in\C$, let
\[
 \mathfrak m_\alpha=(s-\alpha)\subset\C[s]
\]
be the maximal ideal corresponding to $\alpha$.  A subscript
$\mathfrak m_\alpha$ means scalar
localization over $\C[s]$; for example,
\[
 \bigl(\shD_Y[s]\cM_0\cdot f^{s+k}\bigr)_{\mathfrak m_\alpha}
 =\shD_Y[s]_{\mathfrak m_\alpha}\cM_0\cdot f^{s+k}.
\]
We use the following form of the extension and comparison theorems in
\cite{Wu}.

\begin{theorem}[Wu]\label{thm:wu-preliminaries}
Fix $\beta\in\C$.  For every sufficiently large integer $K$, the natural
morphism from the minimal to the maximal extension is injective and
\begin{align}
 j_*\bigl(\cM_U[s]_{\mathfrak m_{-\beta}}\cdot f^s\bigr)
 &=\shD_Y[s]_{\mathfrak m_{-\beta}}\cM_0\cdot f^{s-K},
 \label{eq:wu-jstar}\\
 j_!\bigl(\cM_U[s]_{\mathfrak m_{-\beta}}\cdot f^s\bigr)
 &=\shD_Y[s]_{\mathfrak m_{-\beta}}\cM_0\cdot f^{s+K}.
 \label{eq:wu-jshriek}
\end{align}
Consequently the $\beta$-nearby-cycle module of $\cM$ is
\begin{equation}\label{eq:wu-nearby-quotient}
 \Psi_{f,\beta}\cM
 =
 \frac{\shD_Y[s]_{\mathfrak m_{-\beta}}\cM_0\cdot f^{s-K}}
      {\shD_Y[s]_{\mathfrak m_{-\beta}}\cM_0\cdot f^{s+K}}.
\end{equation}
It depends only on $\cM_U$, is a holonomic $\shD_Y$-module supported on $D$,
and is annihilated by some power of $s+\beta$.  If $\cM$ is regular
holonomic, then
\begin{equation}\label{eq:wu-rh-comparison}
 \operatorname{DR}(\Psi_{f,\beta}\cM)
 \simeq i_*\psi_{f,\lambda}\operatorname{DR}(\cM)[-1],
 \qquad \lambda=\exp(2\pi\ii\beta).
\end{equation}
\end{theorem}

Equations \eqref{eq:wu-jstar}--\eqref{eq:wu-jshriek} and the injectivity are
\cite[Theorem 2.10]{Wu}; equation \eqref{eq:wu-nearby-quotient} is
\cite[Definition 2.11]{Wu}; holonomicity and nilpotence are
\cite[Proposition 2.12]{Wu}; and \eqref{eq:wu-rh-comparison} is
\cite[Theorem 1.1]{Wu}.  We also use the classical implication from
Bernstein--Sato roots to nearby-cycle monodromy
\cite{Kashiwara,Malgrange}.

\begin{theorem}
\label{thm:roots-give-monodromy}
Let $f$ be a nonzero regular function on a smooth complex algebraic variety.
If $\xi$ is a root of a local Bernstein--Sato polynomial $b_{f,x}(s)$, then
\[
 \exp(-2\pi\ii\xi)
\]
occurs as a local monodromy eigenvalue of $f$, equivalently as an eigenvalue
on its nearby-cycle complex.  Consequently, the same conclusion holds for
every root of the global Bernstein--Sato polynomial $b_f(s)$.
\end{theorem}

The sign agrees with \eqref{eq:wu-rh-comparison}: if
$\xi\equiv-\beta\pmod{\Z}$, then
\[
 \exp(-2\pi\ii\xi)=\exp(2\pi\ii\beta).
\]
The global assertion follows because the global Bernstein--Sato polynomial
is the least common multiple of the local polynomials.

For the simplicity argument, in the case $\cM_0=\cO_Y$, set
\begin{equation}\label{eq:adjacent-quotient}
 \cQ_k^\beta
 :=
 \frac{\shD_Y[s]_{\mathfrak m_{-\beta}}f^{s-k}}
      {\shD_Y[s]_{\mathfrak m_{-\beta}}f^{s-k+1}},
 \qquad k\in\Z.
\end{equation}
Replacing the variable in
\eqref{eq:wus-b-quotient} by $s-k$ shows that $b_f(s-k)$ is
the minimal polynomial of $s$ on the corresponding unlocalized quotient.
Localizing at $\mathfrak m_{-\beta}$ therefore gives
\begin{equation}\label{eq:quotient-detects-root}
 \cQ_k^\beta\neq0
 \quad\Longleftrightarrow\quad
 b_f(-\beta-k)=0.
\end{equation}
Thus keeping the localization point $s=-\beta$ fixed and changing the exponent
from $f^s$ to $f^{s-k}$ records exactly the integral translates of a root.

\subsection{The secant stratification and its local product structure}

Let $C\simeq\P^1\subset\P^{2m-2}$ be the rational normal curve of degree
$2m-2$.  For $1\leq a\leq m-1$, let
$S_a=S_a(2m-2)$ denote the closure of the union of the projective
$(a-1)$-planes spanned by $a$ points of $C$, and write
$X_a=X_a(2m-2)$ for its affine cone.  We also set
\[
X_{-1}=\varnothing,
\qquad
X_0=\{0\},
\qquad
X_{m-1}=D_m.
\]
Thus $X_a$ is the affine locus of Hankel rank at most $a$.  In particular,
$S_{m-1}$ is the projective hypersurface defined by $f_m$, and its affine
cone is $X_{m-1}=D_m$.

We use the following consequences of Brogan's Lemma 3.1, Corollary 3.2, and
Theorems 3.17 and 3.22 \cite{Brogan}.  Here $\psi_{f,\lambda}$
denotes the generalized $\lambda$-eigenspace of the nearby-cycle functor.
The symbol $\IC(L)$ denotes the intersection complex, namely the perverse
intermediate extension of a local system $L$ from a dense smooth open set, and
$\Supp$ denotes support.  These conventions agree with
\cite[Section 3.1]{Brogan}.  Under the indexing $n=m-1$, the ambient
dimension $2n+1$ in that paper is $N=2m-1$ here, and the symbols $H_n$ and
$f=\det H_n$ there correspond to $H_m$ and $f_m$ here.  The cited results are
stated with rational coefficients; whenever a generalized
$\lambda$-eigenspace is written below, scalar extension from $\Q$ to $\C$
is understood.

\begin{theorem}[Brogan]\label{thm:brogan}
The following statements hold for $f_m$.
\begin{enumerate}[label=\textup{(\roman*)}]
\item\label{item:brogan-slice}
For $1\leq q\leq m$, at a general point of
$X_{m-q}\setminus X_{m-q-1}$, the transverse analytic
germ of $f_m$ is analytically equivalent, up to multiplication by a unit, to
$f_q$ in the transverse variables; the remaining coordinates are smooth
parameters.
\item\label{item:brogan-eigen}
The eigenvalues of the nearby-cycle monodromy are precisely roots of unity of
orders $q\in\{1,\ldots,m\}$.
\item\label{item:brogan-ic}
If $q>1$ and $\lambda$ is a primitive $q$-th root of unity, then the
$\lambda$-nearby cycles form an intersection complex
\[
\psi_{f_m,\lambda}\Q_{\C^N}[N]
\simeq \IC(L_\lambda),
\qquad
\Supp\IC(L_\lambda)=X_{m-q},
\]
where $L_\lambda$ is an irreducible rank-one local system on a dense smooth
open subset of the support.
\item\label{item:brogan-one}
For $\lambda=1$,
\[
\psi_{f_m,1}\Q_{\C^N}[N]
\simeq\Q_{X_{m-1}}[N-1]\simeq\IC_{X_{m-1}}.
\]
\end{enumerate}
\end{theorem}

To pass from Theorem~\ref{thm:brogan} to $\shD$-modules, we use the regular
Riemann--Hilbert correspondence.  For a smooth complex algebraic variety, it
identifies regular holonomic
$\shD$-modules, via the de Rham functor with its usual normalization, with
perverse sheaves \cite[Theorem 7.2.1]{HTT}.  In particular, it preserves and
reflects simple objects:
a regular holonomic $\shD$-module is simple if and only if its de Rham complex
is a simple perverse sheaf.

By \cite[Theorem 3.17]{Brogan}, every secant variety in question is an
irreducible rational homology manifold.  Hence the shifted constant sheaf in
part~\ref{item:brogan-one} is its intersection complex.  For $q>1$, a
rank-one local system is irreducible, and its intermediate extension is a
simple perverse sheaf.  Thus parts~\ref{item:brogan-ic} and
~\ref{item:brogan-one} identify each nearby-cycle eigensummand with a simple
perverse sheaf.  Applying the regular Riemann--Hilbert correspondence, we
conclude that every monodromy eigensummand corresponds to a simple regular
holonomic $\shD$-module.

The local statement in part~\ref{item:brogan-slice} follows from the explicit
Hankel elimination in \cite[Lemma 3.1]{Brogan}.  On the chart used there, a
unit multiple of $f_m$ becomes a smaller Hankel determinant in the transverse
variables.  The $SL_2(\C)$-action on
$\operatorname{Sym}^{2m-2}(\C^2)$ preserves the rational normal curve and its
secant stratification.  Given a projective point, one may choose an osculating
hyperplane not containing it and then carry that hyperplane to the standard
one by $SL_2(\C)$.  This is the chart argument used in the proof of
\cite[Corollary 3.2]{Brogan}; it does not require the $SL_2$-action to be
transitive on a secant stratum.

\subsection{The unique integral translate}

Combining the lattice filtration with the simplicity of the nearby-cycle
summands in Theorem~\ref{thm:brogan} gives the following restriction on the
roots of the Hankel determinant.

\begin{lemma}
\label{lem:one-jump}
Fix $0<\beta\leq1$.  If $\Psi_{f,\beta}\cO_Y$ is a simple regular holonomic
$\shD_Y$-module, then among the numbers
\[
-\beta-k,\qquad k\in\Z,
\]
at most one is a root of $b_f(s)$.  Any such root has multiplicity one.
\end{lemma}

\begin{proof}
Choose $K$ so that Theorem~\ref{thm:wu-preliminaries} applies and every root
in the congruence class of $-\beta$ is represented by an
integer $k$ with $-K<k\leq K$.  The images of
\[
\shD_Y[s]_{\mathfrak m_{-\beta}}f^{s+K}
\subset\shD_Y[s]_{\mathfrak m_{-\beta}}f^{s+K-1}
\subset\cdots\subset
\shD_Y[s]_{\mathfrak m_{-\beta}}f^{s-K}
\]
in \eqref{eq:wu-nearby-quotient} form a filtration by $\shD_Y$-submodules.
Simplicity forces each term to be either $0$ or the whole nearby-cycle module.
Hence the filtration can jump only once.  Indeed, if two of its successive
quotients were nonzero, a filtration term after the first jump and before the
second would be a nonzero proper submodule.
By \eqref{eq:quotient-detects-root}, at most one integer translate can
therefore be a root.

Suppose that this root has multiplicity $d$.  The unique nonzero adjacent
quotient $\cQ_k^\beta$ is then the whole nearby-cycle module.  To relate the
root multiplicity to the action of $s$, recall that before localization the
minimal polynomial on the corresponding adjacent quotient is $b_f(s-k)$.
Since $-\beta-k$ is a root of $b_f$ of multiplicity $d$, one may write
\[
 b_f(s-k)=(s+\beta)^d u(s),
 \qquad u(-\beta)\neq0.
\]
After localization at $\mathfrak m_{-\beta}$, the factor $u(s)$ is a unit.
Thus the minimal polynomial of $s$ on $\cQ_k^\beta$ is precisely
$(s+\beta)^d$.  By
\cite[Proposition 2.12]{Wu}, the endomorphism $s+\beta$ is nilpotent.  If it
were nonzero, then $(s+\beta)\cQ_k^\beta$ would be a nonzero proper
$\shD_Y$-submodule: it is nonzero by assumption and proper because a
nilpotent endomorphism cannot be surjective on a nonzero module.  This again
contradicts simplicity.  Thus $s+\beta$ acts by zero, its minimal polynomial
is $s+\beta$, and $d=1$.
\end{proof}

\begin{coro}\label{cor:one-root-class}
For $f_m$, every monodromy congruence class contains at most one root of
$b_{f_m}(s)$, and every root is simple.
\end{coro}

\begin{proof}
Apply Theorem~\ref{thm:wu-preliminaries} with $Y=\C^N$, $f=f_m$,
$\cM=\cO_Y$, $\cM_U=\cO_U$, and the coherent lattice $\cM_0=\cO_Y$.
$\cO_Y$ is regular holonomic, and nearby cycles preserve regular
holonomicity.  By \eqref{eq:wu-rh-comparison},
Theorem~\ref{thm:brogan}, and the regular Riemann--Hilbert correspondence,
each nearby-cycle eigensummand corresponds to a simple regular holonomic
$\shD_Y$-module.  The conclusion now follows from Lemma~\ref{lem:one-jump}.
\end{proof}

Here the class corresponding to an eigenvalue $\lambda$ means the set of
rational numbers $\xi$ satisfying
$\exp(-2\pi\ii\xi)=\lambda$.

\subsection{The secant-stratum log resolution and Lichtin's bound}

The zeta calculation determines a root only up to an integral shift.  Lichtin's
theorem removes this ambiguity once the numerical data of a log resolution are
known.  We obtain these data from the affine-cone lift of Bertram's projective
secant resolution, using Brogan's local Hankel reduction to compute the
transverse multiplicities.  Throughout this subsection $m\geq2$; the case
$m=1$ is smooth.

A log resolution of a pair
$(Y,D)$ is a proper birational morphism from a smooth variety which is an
isomorphism over $Y\setminus D$ and for which the strict transform of $D$
together with all exceptional divisors has simple normal crossings.

We use Lichtin's theorem in the following form.

\begin{theorem}[Lichtin]\label{thm:lichtin}
Let $Y$ be smooth, let $f$ be a nonzero reduced regular function, and let
$\mu:\wt Y\to Y$ be a log resolution such that
\[
 \mu^*(f=0)=\sum_i a_iE_i,
 \qquad
 K_{\wt Y/Y}=\sum_i b_iE_i,
\]
where the strict transform is included with $a_i=1$ and $b_i=0$.  Then every
root of $b_f(s)$ is of the form
\begin{equation}\label{eq:lichtin}
 -\frac{b_i+1+\ell}{a_i}
 \qquad(\ell\in\Z_{\geq0}).
\end{equation}
\end{theorem}

This is \cite[Theorem 5]{Lichtin}; see also
\cite[Theorem 3.5.1]{PopaNotes}.  Thus we must compute the multiplicity $a_i$
and discrepancy coefficient $b_i$ of each divisor on a log resolution.  Here
$K_{\wt Y/Y}$ is the relative canonical divisor: locally, $b_i$ is the
order along $E_i$ of the Jacobian determinant of $\mu$.

Since $S_a=\P(X_a)\subset\P^{N-1}$, it is the secant variety
swept out by the linear spans of effective divisors of degree $a$ on the
rational normal curve $C\simeq\P^1$, and $S_{m-1}$ is the projective
hypersurface whose affine cone is $D_m=X_{m-1}$.
For the global resolution we use Bertram's theorem.

\begin{theorem}[Bertram's secant resolution]\label{thm:bertram-resolution}
For the rational normal curve $C\subset\P^{2m-2}$, successively blowing up
the strict transforms of
\[
 S_1\subset S_2\subset\cdots\subset S_{m-2}
\]
gives a log resolution of
$(\P^{2m-2},S_{m-1})$.  At each stage the center is smooth, the final
strict transform of $S_{m-1}$ is smooth, and it meets all exceptional
divisors with simple normal crossings.
\end{theorem}

This is a consequence of Bertram's construction; see
\cite[Propositions 2.2--2.3 and Corollary 2.4]{Bertram}.  A detailed
formulation, including the simple-normal-crossing property, is given in
\cite[Proposition 3.2, Remark 3.3, and Corollary 3.4]{SchnellYang}.  In the
notation of the latter reference, take $d=2m-2$ and $k=m-2$.  The required
positivity hypothesis holds because $\cO_{\P^1}(d)$ separates
$d+1=2m-1$ points, and hence in particular $2k+2=2m-2$ points.

We start with the blow-up of the origin in $\C^N$.  On the standard chart
in which the projective direction labelled by $x_0$ is nonzero, put
\begin{equation}\label{eq:vertex-blowup-chart}
 t=x_0,
 \qquad y_j=\frac{x_j}{x_0}\quad(1\leq j\leq N-1),
 \qquad y_0=1.
\end{equation}
The blow-down map on this chart is
\[
 (t,y_1,\ldots,y_{N-1})
 \longmapsto
 (t,ty_1,\ldots,ty_{N-1}).
\]
The phrase ``$x_0\neq0$ chart'' refers to the corresponding projective
direction; the pullback coordinate $t=x_0$ is allowed to vanish, and
$E_m$ is precisely $(t=0)$.  Since $f_m$ is homogeneous of degree $m$,
\begin{equation}\label{eq:vertex-total-transform}
 \begin{aligned}
 f_m(t,ty_1,\ldots,ty_{N-1})
 &=t^m F_m(y),\\
 F_m(y)&=\det(y_{i+j})_{0\leq i,j\leq m-1},\qquad y_0=1.
 \end{aligned}
\end{equation}
Thus the exceptional divisor occurs in the total transform with multiplicity
$m$, while $F_m=0$ is the strict transform.  Its intersection with $E_m$ is
the $x_0\neq0$ affine part of the projective secant hypersurface.
To compute the discrepancy, differentiate the blow-down map:
\[
 \dd x_0=\dd t,
 \qquad
 \dd x_j=y_j\dd t+t\dd y_j
 \quad(1\leq j\leq N-1).
\]
In the wedge product, every term containing $y_j\dd t$ vanishes because
the first factor is already $\dd t$.  Hence
\begin{align}
 \dd x_0\wedge\cdots\wedge\dd x_{N-1}
 &=\dd t\wedge
   \bigwedge_{j=1}^{N-1}(y_j\dd t+t\dd y_j)\notag\\
 &=t^{N-1}\dd t\wedge\dd y_1\wedge\cdots
   \wedge\dd y_{N-1}.
   \label{eq:vertex-jacobian}
\end{align}
Consequently
\[
 \ord_{E_m}(f_m)=m,
 \qquad
 \ord_{E_m}(K_{\operatorname{Bl}_0\C^N/\C^N})=N-1=2m-2.
\]

Although \eqref{eq:vertex-total-transform} separates the radial factor $t^m$,
it does not identify $F_m$ with $f_{m-1}$.  The required local identification
is due to Brogan.

\begin{prop}[Brogan]
\label{prop:brogan-first-blowup}
For every point $p\in S_{m-1}$ there is a Zariski neighborhood
$U_p\subset S_{m-1}$ of $p$ and an isomorphism of hypersurface germs
\begin{equation}\label{eq:brogan-projective-local-product}
 (U_p,p)\simeq
 \bigl(\A^1\times D_{m-1},(a,p')\bigr).
\end{equation}
Here the left-hand side denotes the germ of the hypersurface $S_{m-1}$ at
$p$, not the germ of the ambient projective space.
Consequently, the strict transform of $D_m$ after blowing up the origin
has the local form
\begin{equation}\label{eq:brogan-affine-blowup-local-product}
 \wt D_m\simeq
 \A^1_t\times\A^1\times D_{m-1}.
\end{equation}
In suitable ambient coordinates, its defining equation is $f_{m-1}$,
independent of the two smooth coordinates.
\end{prop}

\begin{proof}
The projective statement is proved in \cite[Corollary 3.2]{Brogan}.  With
$n=m-1$, the calculation on the chart $x_0\neq0$ is the case $k=0$ of
\cite[Lemma 3.1]{Brogan}.  It gives a triangular change
of coordinates for which the dehomogenized equation is $f_{m-1}$ and one
coordinate is absent.  Corollary~3.2 uses the $SL_2(\C)$-action to cover every
projective point by such a chart.  Adding the fiber coordinate $t$ from
\eqref{eq:vertex-blowup-chart} gives
\eqref{eq:brogan-affine-blowup-local-product}.
\end{proof}

Brogan's proposition supplies the local recursive model for the centers
appearing in the affine-cone resolution.
On the chart just described, the pair consisting of the ambient space and the
strict transform is a smooth two-dimensional factor times the pair for
$f_{m-1}$.  The next center is therefore the smooth factor times the vertex of
$f_{m-1}$.  Since blow-up commutes with smooth base change, its local model is
the blow-up of that smaller vertex, with the smooth coordinates unchanged.
Repeating gives the chain
\[
 f_m\rightsquigarrow f_{m-1}\rightsquigarrow f_{m-2}
 \rightsquigarrow\cdots.
\]
For $m=2$ the smaller determinant is linear, so the recursion stops.

\begin{lemma}
\label{lem:affine-secant-resolution}
The affine-cone counterpart of Theorem~\ref{thm:bertram-resolution} is obtained
as follows: first blow up the origin and then successively blow up the strict
transforms of
\[
 X_1\subset X_2\subset\cdots\subset X_{m-2}.
\]
The resulting morphism is a log resolution of $(\C^N,D_m)$.  After the initial
blow-up, the remaining stages are obtained from Bertram's projective resolution
by smooth base change along the projection of the tautological line bundle.
When $m=2$, the displayed sequence of further centers is empty.
\end{lemma}

\begin{proof}
The blow-up of the origin and the strict transform of the affine cone are
\[
 \operatorname{Bl}_0\C^N
 \simeq\operatorname{Tot}\bigl(\cO_{\P^{N-1}}(-1)\bigr),
 \qquad
 \wt D_m
 \simeq\operatorname{Tot}\bigl(\cO_{S_{m-1}}(-1)\bigr).
\]
Here $\operatorname{Tot}$ denotes the total space of a line bundle.  If
\[
 q_0:\operatorname{Tot}\bigl(\cO_{\P^{N-1}}(-1)\bigr)
 \longrightarrow\P^{N-1}
\]
is its projection, then the strict transform of the cone $X_a$ is
$q_0^{-1}(S_a)=\operatorname{Tot}(\cO_{S_a}(-1))$.  In particular, the
first center after the vertex blow-up is the inverse image of
$S_1\simeq\P^1$.

The morphism $q_0$ is smooth, and blow-up commutes with smooth base change.
Pulling back each stage of the projective sequence in
Theorem~\ref{thm:bertram-resolution} therefore gives the blow-up of
the corresponding strict transform of $X_a$.  Inductively, every resulting
ambient space is the total space of the pullback of
$\cO_{\P^{N-1}}(-1)$ to the corresponding stage of Bertram's resolution.

The final projective ambient space is smooth, and the strict transform of
$S_{m-1}$ together with the projective exceptional divisors has simple normal
crossings.  These properties persist under smooth base change.  It remains
only to include the exceptional divisor of the initial vertex blow-up, which
is the zero section of the pulled-back line bundle.  In a local
trivialization the total space is $V\times\A^1$, the pulled-back divisors come
from $V$, and the zero section is given by the fiber coordinate $t=0$.
Consequently it meets every pulled-back component transversely, so the full
total transform is again a simple-normal-crossing divisor.  Finally, all
centers lie over $D_m$, and the composition of the blow-ups is proper and is
an isomorphism over $\C^N\setminus D_m$.  Hence the affine sequence is a log
resolution.
\end{proof}

Denote this affine-cone resolution by
\[
 \mu:\wt Y_m\longrightarrow\C^N.
\]
For $2\leq q\leq m$, let $E_q$ be the divisor created from
$X_{m-q}$, and put $E_1=\wt D_m$.  At a general point of
$X_{m-q}\setminus X_{m-q-1}$, the earlier blow-ups are
isomorphisms, and Theorem~\ref{thm:brogan}\ref{item:brogan-slice} identifies
the normal pair with the vertex pair for $f_q$.  The normal space has dimension
\[
 c_q=N-2(m-q)=2q-1.
\]
The calculation
\eqref{eq:vertex-total-transform}--\eqref{eq:vertex-jacobian}, now with
degree $q$ and ambient dimension $c_q$, therefore applies in the normal
directions and gives
\[
 a_q:=\ord_{E_q}(f_m)=q,
 \qquad
 b_q:=\ord_{E_q}(K_{\wt Y_m/\C^N})=c_q-1=2q-2.
\]
Because these are divisorial orders, their values at the general point
determine them globally.  For the strict transform, $a_1=1$ and $b_1=0$.
Thus
\begin{equation}\label{eq:resolution-data}
\mu^*D_m=\sum_{q=1}^m qE_q,
\qquad
K_{\wt Y_m/\C^N}=\sum_{q=2}^m(2q-2)E_q.
\end{equation}

Applying Theorem~\ref{thm:lichtin} to \eqref{eq:resolution-data} yields the
following resolution bound.

\begin{coro}
\label{cor:resolution-bound}
Every root of the global polynomial $b_{f_m}(s)$, and every root of the local
polynomial $b_{f_m,0}(s)$ at the vertex, belongs to
\begin{equation}\label{eq:resolution-candidates}
 \left\{-\frac{2q-1+\ell}{q}:
 1\leq q\leq m,\ \ell\in\Z_{\geq0}\right\}.
\end{equation}
Moreover, if $-\alpha$ is such a root and its monodromy eigenvalue has exact
order $m$, then
\begin{equation}\label{eq:N-over-m}
 \alpha\geq\frac Nm=2-\frac1m.
\end{equation}
\end{coro}

\begin{proof}
Lemma~\ref{lem:affine-secant-resolution} shows that $\mu$ is a log resolution.
Moreover, $D_m$ is the irreducible secant hypersurface in
Theorem~\ref{thm:brogan}; hence its
defining polynomial $f_m$ is reduced.  Thus every hypothesis of
Theorem~\ref{thm:lichtin} is satisfied.  For $E_q$, equations
\eqref{eq:resolution-data} give
$a_q=q$ and $b_q=2q-2$.  Substitution in \eqref{eq:lichtin} gives
\eqref{eq:resolution-candidates}.  Restricting $\mu$ over a neighborhood of
the origin is again a log resolution, so the same argument gives the stated
local inclusion.

Finally, if the eigenvalue of $-\alpha$ has exact order $m$, then $\alpha$
has reduced denominator $m$.  A number represented by the $q$-th set in
\eqref{eq:resolution-candidates} has reduced denominator dividing $q$.
Since $q\leq m$, this forces $q=m$, and hence
\[
 \alpha=\frac{2m-1+\ell}{m}\geq\frac{2m-1}{m}
 =\frac Nm=2-\frac1m.
\]
\end{proof}

Corollary~\ref{cor:resolution-bound} only bounds the possible roots; a log
resolution does not show that the corresponding residues are nonzero.  That
nonvanishing will come from the zeta integral.  Homogeneity will then identify
the local polynomial at the vertex with the global polynomial.

Corollaries~\ref{cor:one-root-class} and~\ref{cor:resolution-bound}, together
with the low-rank calculations in Section~\ref{sec:checks}, suggest the product
in \eqref{eq:main-formula}: one simple root in each primitive monodromy class,
at the translate selected by the resolution bound.  What remains is to prove
that these candidate roots exist.  Sections~\ref{sec:cone} and~\ref{sec:zeta}
construct and evaluate a branch zeta integral with nonvanishing primitive
residues, and Section~\ref{sec:poles-roots} fixes their exact translates.

\section{The positive Hankel cone and Gauss--Radau coordinates}
\label{sec:cone}

Throughout this section $m\geq2$.  Set
\begin{equation}\label{eq:r-def}
r=m-1,
\qquad\text{so that}\qquad N=2r+1.
\end{equation}

\subsection{Truncated moments and a coercive linear form}

The entries of a Hankel matrix depend only on the sum of their indices.  This
is the pattern of a moment Gram matrix: the $(p,q)$ entry is the value
of a linear functional on $u^{p+q}$.  Introducing that functional converts
matrix positivity into positivity of an inner product and makes spectral
theory available.

For $x=(x_0,\ldots,x_{2r})\in\R^{2r+1}$, define a linear functional on
polynomials of degree at most $2r$ by
\begin{equation}\label{eq:Lx}
L_x(u^j)=x_j.
\end{equation}
Here $u$ is a formal indeterminate in the polynomial ring $\R[u]$.  By
contrast, the symbols $u_a$ introduced below denote real nodes; together with
the weights $w_a$, they serve as coordinates in the subsequent change of
variables.
On the space $\R[u]_{\leq r}$, the bilinear form
\begin{equation}\label{eq:moment-inner}
\langle p,q\rangle_x=L_x(pq)
\end{equation}
has Gram matrix $H_m(x)$ in the monomial basis.  Consequently
\begin{equation}\label{eq:omega}
\Omega_m=\{x\in\R^{2r+1}:H_m(x)>0\}
\end{equation}
is an open convex cone.  It is a connected component of $(f_m>0)$: inertia is
locally constant on the space of nonsingular real symmetric matrices, and the
component containing a positive definite matrix consists precisely of positive
definite matrices.  Its closure is the positive semidefinite Hankel cone.  Indeed,
one inclusion follows by continuity; conversely, adding a positive multiple of
any fixed positive definite Hankel matrix approximates every positive semidefinite
Hankel matrix by points of $\Omega_m$.

To turn the cone integral into a radial Mellin integral, we need a linear
functional that is positive on every nonzero boundary ray.  The
$SO(2)\subset SL_2(\R)$-invariant binary form $(X^2+Y^2)^r$ suggests the
following choice after setting $u=X/Y$:
\begin{equation}\label{eq:ell}
\ell_m(x)=\sum_{j=0}^{r}\binom rj x_{2j}
=L_x((1+u^2)^r).
\end{equation}

\begin{lemma}\label{lem:ell-coercive}
The functional $\ell_m$ is strictly positive on
$\wb\Omega_m\setminus\{0\}$.  Moreover, there is a constant $C_m>0$
such that
\begin{equation}\label{eq:ell-norm}
\|x\|\leq C_m\ell_m(x)
\qquad(x\in\wb\Omega_m).
\end{equation}
In particular,
$\wb\Omega_m\cap\{\ell_m\leq R\}$ is compact for every $R>0$.
\end{lemma}

\begin{proof}
Recall that a real symmetric matrix is positive semidefinite if and only if
all of its principal minors are nonnegative.  The principal minors of order
one therefore give
\[
H_{aa}=x_{2a}\geq0,\qquad 0\leq a\leq r.
\]
Every coefficient in \eqref{eq:ell} is positive.  Hence
$\ell_m(x)=0$ implies $H_{aa}=0$ for every $a$.  For $a\neq b$, the principal
minor on rows and columns $a,b$ now satisfies
\[
0\leq
\det\begin{pmatrix}H_{aa}&H_{ab}\\ H_{ab}&H_{bb}\end{pmatrix}
=H_{aa}H_{bb}-H_{ab}^2=-H_{ab}^2.
\]
Thus $H_{ab}=0$ as well.  Every entry of $H_m(x)$ vanishes, and since the
entries of a Hankel matrix contain all the coordinates
$x_0,\ldots,x_{2r}$, it follows that $x=0$.  Thus $\ell_m$ is positive on the
compact set
$\wb\Omega_m\cap\{\|x\|=1\}$.  Its positive minimum on that set gives
\eqref{eq:ell-norm} by homogeneity.
\end{proof}

The choice $(1+u^2)^r$ in \eqref{eq:ell} serves both parts of the argument.
Its positivity gives Lemma~\ref{lem:ell-coercive}; under the later substitution
$u=\tan(\theta/2)$, it cancels all factors involving
$|1+\e^{\ii\theta}|$ and leaves a Morris integral.

\subsection{Construction of nodes and weights}

We construct real nodes and positive weights, with one node fixed at zero,
directly from the positive definite inner product \eqref{eq:moment-inner}.
No quadrature or measure representation is assumed.  The construction is
unique away from a proper algebraic subset.

\begin{theorem}\label{thm:radau}
Outside a real algebraic subset of measure zero, every $x\in\Omega_m$ has a
unique unordered representation
\begin{equation}\label{eq:moment-param}
 x_j=\sum_{a=0}^{r}w_a u_a^j,
 \qquad 0\leq j\leq2r.
\end{equation}
Here
\begin{equation}\label{eq:param-domain}
u_0=0,
\quad u_1,\ldots,u_r\in\R\setminus\{0\}
\text{ are pairwise distinct},
\quad w_0,\ldots,w_r>0.
\end{equation}
Without ordering the nonzero nodes, the parameter map is generically
$r!$-to-one.
\end{theorem}

\begin{proof}
Fix $x\in\Omega_m$ and consider the positive definite inner product
\eqref{eq:moment-inner} on
$\mathscr P_r=\R[u]_{\leq r}$.
The subspace $u\mathscr P_{r-1}$ has codimension one.
Choose a nonzero $h\in\mathscr P_r$ such that
\begin{equation}\label{eq:h-orthogonal}
h\perp u\mathscr P_{r-1}.
\end{equation}
The orthogonal complement is one-dimensional, since
\[
\dim\mathscr P_r=r+1,
\qquad
\dim(u\mathscr P_{r-1})=r.
\]
Thus $h$ is unique up to a nonzero scalar.  Before normalizing it, we must
verify that it has degree $r$ outside an exceptional algebraic subset.
Write
\[
 h(u)=\sum_{\ell=0}^{r}c_\ell u^\ell.
\]
Since $u\mathscr P_{r-1}$ has basis $u,u^2,\ldots,u^r$, the orthogonality
conditions in \eqref{eq:h-orthogonal} are explicitly
\[
 0=\langle h,u^j\rangle_x
   =L_x(hu^j)
   =\sum_{\ell=0}^{r}x_{j+\ell}c_\ell,
 \qquad 1\leq j\leq r.
\]
Thus they form a homogeneous linear system for the $r+1$ coefficients of
$h$, with coefficient matrix
\[
B(x)=
\begin{pmatrix}
x_1&x_2&\cdots&x_{r+1}\\
x_2&x_3&\cdots&x_{r+2}\\
\vdots&\vdots&\ddots&\vdots\\
x_r&x_{r+1}&\cdots&x_{2r}
\end{pmatrix}.
\]
The solution space is exactly $(u\mathscr P_{r-1})^\perp$, hence is
one-dimensional; equivalently, $B(x)$ has rank $r$ and determines the line
$\R h=\ker B(x)$ from $x$.

For a full-rank $r\times(r+1)$ matrix, the signed maximal minors form a
kernel vector.  It follows that the last coordinate $c_r$ is, up to a common
nonzero factor and sign, the maximal minor obtained by deleting the last
column of $B(x)$.  In the present notation this minor is the shifted moment
determinant
\[
 \Delta_r^{(1)}(x)
 :=\det
 \begin{pmatrix}
 x_1&x_2&\cdots&x_r\\
 x_2&x_3&\cdots&x_{r+1}\\
 \vdots&\vdots&\ddots&\vdots\\
 x_r&x_{r+1}&\cdots&x_{2r-1}
 \end{pmatrix}.
\]
Consequently,
\[
 \deg h<r
 \quad\Longleftrightarrow\quad
 c_r=0
 \quad\Longleftrightarrow\quad
 \Delta_r^{(1)}(x)=0.
\]
Equivalently, this is the degeneracy locus of the shifted moment form
$(p,q)\mapsto L_x(upq)$ on $\mathscr P_{r-1}$.  Thus the exceptional condition
is algebraic in the moments.

To see that $\Delta_r^{(1)}$ is a nonzero polynomial, evaluate it at the
point $x^\star\in\R^{2r+1}$ defined by
\[
 x_r^\star=1,\qquad
 x_j^\star=0\quad(0\leq j\leq2r,\ j\neq r).
\]
The matrix defining $\Delta_r^{(1)}(x^\star)$ has ones on the anti-diagonal
and zeros elsewhere.  It is the $r\times r$ matrix obtained by reversing
the columns of the identity matrix.  Hence
\[
 \Delta_r^{(1)}(x^\star)
 =\det\begin{pmatrix}
  0&\cdots&0&1\\
  0&\cdots&1&0\\
  \vdots&&\vdots&\vdots\\
  1&\cdots&0&0
 \end{pmatrix}
 =(-1)^{r(r-1)/2}\neq0,
\]
since the reversal permutation has $r(r-1)/2$ inversions.  This test point
need not belong to $\Omega_m$: the calculation proves that
$\Delta_r^{(1)}$ is nonzero as a polynomial on the ambient space
$\R^{2r+1}$.  Its zero set is therefore a proper real algebraic subset of
Lebesgue measure zero, and its intersection with $\Omega_m$ also has
measure zero.  Outside this intersection, the preceding equivalence gives
$\deg h=r$, and we may normalize $h$ to be monic.

Then $\mathscr P_r=\mathscr P_{r-1}\oplus\R h$.  Define
\begin{equation}\label{eq:A-def}
Ap=up\quad(p\in\mathscr P_{r-1}),
\qquad Ah=0.
\end{equation}
For $p,q\in\mathscr P_{r-1}$,
\[
\langle Ap,q\rangle_x=L_x(upq)=\langle p,Aq\rangle_x,
\]
and the mixed terms with $h$ vanish by \eqref{eq:h-orthogonal}.  Thus $A$ is
self-adjoint: if $A^*$ denotes the adjoint characterized by
$\langle Ap,q\rangle_x=\langle p,A^*q\rangle_x$, the preceding identities say
that $A^*=A$ on all of $\mathscr P_r$.

We recall the finite-dimensional spectral theorem in the form used here.
After complexification, extend $\langle\ ,\ \rangle_x$ sesquilinearly,
with the first argument linear, to a Hermitian inner product.  The operator
$A$ remains self-adjoint.  On a complex inner-product space, an operator $T$
is unitarily diagonalizable if and only if it is normal, meaning that
$TT^*=T^*T$.  Every self-adjoint operator is normal, since $T=T^*$ implies
$TT^*=T^*T$.  Its
eigenvalues are real: if $Te=\kappa e$ and $e\neq0$, then
\[
 \kappa\langle e,e\rangle
 =\langle Te,e\rangle
 =\langle e,Te\rangle
 =\wb\kappa\langle e,e\rangle.
\]
It follows that $A$ is diagonalizable with real eigenvalues.  Equivalently,
the real spectral theorem provides an orthonormal basis of real eigenvectors
for $A$.

Self-adjointness alone does not imply that these eigenvalues are distinct; the
additional input is cyclicity.  Indeed,
\[
1,A1,\ldots,A^r1=1,u,\ldots,u^r
\]
is a basis of $\mathscr P_r$, so $1$ is cyclic for $A$.  Suppose that the
eigenspace $\ker(A-\kappa I)$ had dimension at least two.  The linear functional
$e\mapsto\langle e,1\rangle_x$ would then have a nonzero kernel on
$\ker(A-\kappa I)$; choose $0\neq e\in\ker(A-\kappa I)$ with
$\langle e,1\rangle_x=0$.  For every $k\geq0$, self-adjointness and
$A^k e=\kappa^k e$ give
\[
 \langle e,A^k1\rangle_x
 =\langle A^k e,1\rangle_x
 =\kappa^k\langle e,1\rangle_x=0.
\]
Thus $e$ would be orthogonal to the cyclic span of $1$, which is all of
$\mathscr P_r$, a contradiction.  Every eigenspace is therefore
one-dimensional.  Since $A$ is diagonalizable on the $(r+1)$-dimensional
space $\mathscr P_r$, it has exactly $r+1$ pairwise distinct eigenvalues.
Finally, $Ah=0$, so one eigenvalue is $u_0=0$; the remaining eigenvalues
$u_1,\ldots,u_r$ are distinct and nonzero.

Choose an orthonormal eigenbasis $e_0,\ldots,e_r$ with $Ae_a=u_a e_a$,
and put
\begin{equation}\label{eq:weights}
w_a=|\langle1,e_a\rangle_x|^2.
\end{equation}
Every $w_a$ is positive, because an eigenvector orthogonal to $1$ would be
orthogonal to the cyclic subspace generated by $1$.  For $0\leq j\leq2r$,
choose nonnegative integers $p,q\leq r$ with $p+q=j$; for example, take
$(p,q)=(j,0)$ if $j\leq r$ and $(p,q)=(r,j-r)$ otherwise.  Since
$A^k1=u^k$ for $0\leq k\leq r$,
\[
x_j=\langle u^p,u^q\rangle_x
=\langle A^p1,A^q1\rangle_x.
\]
Expand the cyclic vector in the orthonormal eigenbasis:
\[
1=\sum_{a=0}^{r}\langle1,e_a\rangle_x e_a,
\qquad
A^p1=\sum_{a=0}^{r}\langle1,e_a\rangle_xu_a^p e_a.
\]
Orthonormality eliminates every cross term, and therefore
\begin{align*}
\langle A^p1,A^q1\rangle_x
&=\sum_{a,b}
\langle1,e_a\rangle_x\wb{\langle1,e_b\rangle_x}
u_a^pu_b^q\langle e_a,e_b\rangle_x\\
&=\sum_{a=0}^{r}|\langle1,e_a\rangle_x|^2u_a^{p+q}
=\sum_{a=0}^{r}w_au_a^j.
\end{align*}
This proves \eqref{eq:moment-param}, and hence existence.

By linearity, any nodes and weights satisfying \eqref{eq:moment-param}
also satisfy
\[
 L_x(g)=\sum_{a=0}^{r}w_ag(u_a)
 \qquad\bigl(g\in\R[u]_{\leq2r}\bigr).
\]
This is a consequence of the moment equations, not an additional
assumption.
For uniqueness, suppose \eqref{eq:moment-param}--\eqref{eq:param-domain} hold.
Equation \eqref{eq:h-orthogonal} yields, for every
$p\in\mathscr P_{r-1}$,
\[
0=\langle h,up\rangle_x=L_x(hup)
=\sum_{a=1}^{r}w_au_ah(u_a)p(u_a).
\]
The last equality uses the preceding identity and
$\deg(hup)\leq2r$.  For a fixed $b\in\{1,\ldots,r\}$, take the Lagrange
polynomial
\[
p_b(u)=\prod_{\substack{1\leq a\leq r\\a\neq b}}
\frac{u-u_a}{u_b-u_a}\in\mathscr P_{r-1}.
\]
It satisfies $p_b(u_a)=\delta_{ab}$, so the preceding identity becomes
\[
0=w_bu_bh(u_b).
\]
The assumptions $w_b>0$ and $u_b\neq0$ give $h(u_b)=0$.  Thus all $r$
nonzero nodes are roots of the monic degree-$r$ polynomial $h$, and hence
\begin{equation}\label{eq:h-roots}
h(u)=\prod_{a=1}^{r}(u-u_a).
\end{equation}
Since the line $\R h$ is determined by $x$ and $h$ has been normalized
to be monic, the polynomial $h$ itself is uniquely determined by $x$.
Consequently its roots are uniquely determined as a multiset.  Since
$u_1,\ldots,u_r$ are pairwise distinct, their unordered set is uniquely
determined by $x$.

Once $u_0=0,u_1,\ldots,u_r$ are fixed, the first $r+1$ moment equations form
the Vandermonde system
\[
\begin{pmatrix}
1 & 1 & \cdots & 1\\
0 & u_1 & \cdots & u_r\\
0 & u_1^2 & \cdots & u_r^2\\
\vdots & \vdots & \ddots & \vdots\\
0 & u_1^r & \cdots & u_r^r
\end{pmatrix}
\begin{pmatrix}
w_0\\ w_1\\ w_2\\ \vdots\\ w_r
\end{pmatrix}
=
\begin{pmatrix}
x_0\\ x_1\\ x_2\\ \vdots\\ x_r
\end{pmatrix}.
\]
Its determinant is
\[
\prod_{0\leq a<b\leq r}(u_b-u_a)\neq0,
\]
because the nodes are pairwise distinct.  The weights are therefore unique.
Only a simultaneous permutation of the pairs
$(u_a,w_a)$, $1\leq a\leq r$, remains possible.  This proves the asserted
unordered uniqueness and the generic $r!$-to-one assertion.
\end{proof}

The identity for $L_x$ obtained in the proof is a quadrature formula exact
through degree $2r$.  With the node $u_0=0$ prescribed, it is of
Gauss--Radau type; this terminology describes the representation just
constructed.

\subsection{Determinant, Jacobian, and the Laplace weight}

We next compute the Hankel determinant, the Jacobian of the moment map, and the
Laplace weight in these parameters.  The Jacobian also shows that the
parameters are smooth coordinates on the required full-measure open set.

Let
\begin{equation}\label{eq:Delta}
\Delta(u)=\prod_{1\leq a<b\leq r}(u_b-u_a).
\end{equation}

\Needspace{15\baselineskip}
\begin{coro}\label{cor:factorization}
Let the nodes and weights satisfy \eqref{eq:param-domain}.  Write
$w=(w_0,\ldots,w_r)$ and $u=(u_1,\ldots,u_r)$, and denote the moment map in
\eqref{eq:moment-param} by $\Phi$, so that $x=\Phi(w,u)$.  Then
\begin{align}
f_m(x)
&=w_0\prod_{a=1}^{r}w_au_a^2\Delta(u)^2,
\label{eq:fm-factor}\\*
|\det D\Phi|
&=\left(\prod_{a=1}^{r}w_au_a^2\right)|\Delta(u)|^4,
\label{eq:jacobian}\\*
\ell_m(x)
&=w_0+\sum_{a=1}^{r}w_a(1+u_a^2)^r.
\label{eq:ell-factor}
\end{align}
\end{coro}

\begin{proof}
Let
\[
V=
\begin{pmatrix}
1&1&\cdots&1\\
0&u_1&\cdots&u_r\\
0&u_1^2&\cdots&u_r^2\\
\vdots&\vdots&\ddots&\vdots\\
0&u_1^r&\cdots&u_r^r
\end{pmatrix}.
\]
Thus $V=(u_a^i)_{0\leq i,a\leq r}$, with its first column simplified by
$u_0=0$.  Then
$H_m=V\operatorname{diag}(w_0,\ldots,w_r)V^{\mathsf T}$ and
\[
\det V=\left(\prod_{a=1}^{r}u_a\right)\Delta(u).
\]
This proves \eqref{eq:fm-factor}.

The derivative matrix of $\Phi$ has columns
\[
\frac{\partial x_j}{\partial w_a}=u_a^j,
\qquad
\frac{\partial x_j}{\partial u_a}=w_aju_a^{j-1}
\quad(a\geq1).
\]
Set
\[
v(t)=(1,t,t^2,\ldots,t^{2r})^{\mathsf T},
\qquad
v'(t)=(0,1,2t,\ldots,2rt^{2r-1})^{\mathsf T}.
\]
In the parameter order $(w_0,\ldots,w_r,u_1,\ldots,u_r)$, the columns are
\[
v(0),v(u_1),\ldots,v(u_r),
w_1v'(u_1),\ldots,w_rv'(u_r).
\]
After extracting $\prod_{a=1}^{r}w_a$ and interleaving the value and
derivative columns, it remains, up to sign, to calculate
\begin{equation}\label{eq:confluent-determinant}
\Delta_{\mathrm{conf}}(u)
=\det\bigl[v(0),v(u_1),v'(u_1),\ldots,v(u_r),v'(u_r)\bigr].
\end{equation}

This confluent Vandermonde determinant can be obtained directly from the
ordinary one.  Replace $v'(u_a)$ temporarily by
\[
\frac{v(u_a+\varepsilon_a)-v(u_a)}{\varepsilon_a}.
\]
Column subtraction shows that the resulting determinant is the ordinary
Vandermonde determinant at the nodes
\[
0,u_1,u_1+\varepsilon_1,\ldots,u_r,u_r+\varepsilon_r
\]
divided by $\varepsilon_1\cdots\varepsilon_r$.  Letting every
$\varepsilon_a$ tend to zero, the difference within the $a$-th pair cancels
the denominator $\varepsilon_a$; the two differences between $0$ and that
pair contribute $u_a^2$; and the four differences between the $a$-th and
$b$-th pairs contribute $(u_b-u_a)^4$.  Hence
\begin{equation}\label{eq:confluent-value}
\Delta_{\mathrm{conf}}(u)=\left(\prod_{a=1}^{r}u_a^2\right)
\prod_{a<b}(u_b-u_a)^4
=\left(\prod_{a=1}^{r}u_a^2\right)\Delta(u)^4.
\end{equation}
There is no extra factorial, since every repeated node contributes only its
first derivative column and $1!=1$.  Restoring the extracted weights and
taking absolute values proves \eqref{eq:jacobian}.

Finally, \eqref{eq:ell} and \eqref{eq:moment-param} give directly
\begin{align*}
\ell_m(x)
&=\sum_{j=0}^{r}\binom rj x_{2j}
=\sum_{a=0}^{r}w_a\sum_{j=0}^{r}\binom rj u_a^{2j}\\
&=\sum_{a=0}^{r}w_a(1+u_a^2)^r\\
&=w_0+\sum_{a=1}^{r}w_a(1+u_a^2)^r,
\end{align*}
because $u_0=0$.  This proves \eqref{eq:ell-factor}.
\end{proof}

It remains to justify the use of these parameters as coordinates.  Any choice
of distinct real nodes $u_0=0,u_1,\ldots,u_r$ and positive
weights defines a point $x\in\Omega_m$ through \eqref{eq:moment-param}:
the factorization $H_m=V\operatorname{diag}(w_0,\ldots,w_r)V^{\mathsf T}$ 
is positive definite, and the matrix $V$ is invertible.
Moreover, \eqref{eq:jacobian} gives $\det D\Phi\neq0$ throughout this
parameter domain, so the inverse function theorem makes $\Phi$ a local
diffeomorphism.  Let
$\Omega_m^\circ=\Omega_m\setminus\{\Delta_r^{(1)}=0\}$.
By Theorem~\ref{thm:radau}, restricting to $\Phi^{-1}(\Omega_m^\circ)$ and
imposing $u_1<\cdots<u_r$ gives a bijection onto $\Omega_m^\circ$.  Its local
inverses are smooth, so it is a diffeomorphism.  Allowing all orders gives an
$r!$-sheeted
covering of $\Omega_m^\circ$.  The parameters mapping into the measure-zero
set $\Omega_m\setminus\Omega_m^\circ$ also form a measure-zero set because
$\Phi$ is a local diffeomorphism.  These facts justify the change of variables
below.

\begin{example}
\label{ex:m3-coordinates}

For $m=3$ the construction can be written out in full.

Let $m=3$, so $r=2$, and write
\[
H_3(x)=
\begin{pmatrix}
x_0&x_1&x_2\\
x_1&x_2&x_3\\
x_2&x_3&x_4
\end{pmatrix}>0.
\]
On the generic chart under consideration, the monic polynomial
\[
h(u)=u^2+c_1u+c_0
\]
is determined by $h\perp u\mathscr P_1$.  The two orthogonality equations are
\[
L_x(uh)=x_3+c_1x_2+c_0x_1=0,
\qquad
L_x(u^2h)=x_4+c_1x_3+c_0x_2=0,
\]
or, equivalently,
\begin{equation}\label{eq:m3-h-system}
\begin{pmatrix}x_1&x_2\\x_2&x_3\end{pmatrix}
\begin{pmatrix}c_0\\c_1\end{pmatrix}
=-\begin{pmatrix}x_3\\x_4\end{pmatrix}.
\end{equation}
Away from the algebraic locus $x_1x_3-x_2^2=0$, this system has a unique
solution.  The two nonzero Radau nodes are the roots of $h$:
\[
h(u)=(u-u_1)(u-u_2).
\]
Once the nodes are known, the weights are determined by
\[
\begin{pmatrix}
1&1&1\\0&u_1&u_2\\0&u_1^2&u_2^2
\end{pmatrix}
\begin{pmatrix}w_0\\w_1\\w_2\end{pmatrix}
=\begin{pmatrix}x_0\\x_1\\x_2\end{pmatrix}.
\]
Explicitly,
\begin{equation}\label{eq:m3-weights}
w_1=\frac{x_1u_2-x_2}{u_1(u_2-u_1)},
\qquad
w_2=\frac{x_2-x_1u_1}{u_2(u_2-u_1)},
\qquad
w_0=x_0-w_1-w_2.
\end{equation}
The positivity of these expressions is not evident from the formulas alone;
it is supplied by the spectral construction in the proof of
Theorem~\ref{thm:radau}.

The Vandermonde matrix and the structured congruence factorization are
\[
V=
\begin{pmatrix}
1&1&1\\0&u_1&u_2\\0&u_1^2&u_2^2
\end{pmatrix},
\qquad
H_3=V\operatorname{diag}(w_0,w_1,w_2)V^{\mathsf T}.
\]
Since $\det V=u_1u_2(u_2-u_1)$, we obtain
\begin{equation}\label{eq:m3-f}
f_3(x)=w_0w_1w_2u_1^2u_2^2(u_2-u_1)^2.
\end{equation}
With parameter order $(w_0,w_1,w_2,u_1,u_2)$, the full Jacobian matrix is
\[
D\Phi=
\begin{pmatrix}
1&1&1&0&0\\
0&u_1&u_2&w_1&w_2\\
0&u_1^2&u_2^2&2w_1u_1&2w_2u_2\\
0&u_1^3&u_2^3&3w_1u_1^2&3w_2u_2^2\\
0&u_1^4&u_2^4&4w_1u_1^3&4w_2u_2^3
\end{pmatrix}.
\]
The calculation \eqref{eq:confluent-determinant}--\eqref{eq:confluent-value}
specializes to
\begin{equation}\label{eq:m3-jacobian}
|\det D\Phi|
=w_1w_2u_1^2u_2^2|u_2-u_1|^4.
\end{equation}
Finally,
\begin{align}
\ell_3(x)
&=x_0+2x_2+x_4\notag\\
&=w_0+w_1(1+u_1^2)^2+w_2(1+u_2^2)^2.
\label{eq:m3-ell}
\end{align}
Thus every factor used in the zeta calculation is already visible in this
$3\times3$ case.
\end{example}

\section{Evaluation of the Laplace zeta integral}\label{sec:zeta}

\subsection{The radial candidate poles}

For $\operatorname{Re}s>0$, define
\begin{equation}\label{eq:Z-def}
\cZ_m(s)=\int_{\Omega_m}
\e^{-\ell_m(x)}f_m(x)^s\dd x.
\end{equation}
Here and below $\dd x=\dd x_0\cdots\dd x_{2m-2}$ is Lebesgue measure.  This
is an Archimedean local zeta integral attached to the branch $\Omega_m$, with
a Laplace weight.  The condition $f_m>0$ makes $f_m^s$ single-valued, while
the exponential ensures convergence on the noncompact cone and separates the
radial scaling.  The integral is absolutely convergent: the exponential
controls infinity by Lemma~\ref{lem:ell-coercive}, whereas
$f_m^{\operatorname{Re}s}$ is locally bounded and vanishes at the boundary of
the cone.  We initially work in this half-plane, so absolute convergence
justifies every use of Fubini's theorem and every change of variables below.

Homogeneity already displays the candidate poles produced by the vertex.
Write
\[
x=\rho y,
\qquad
\rho=\ell_m(x)>0,
\qquad
y\in\Sigma_m:=\Omega_m\cap(\ell_m=1).
\]
The map
\[
\Theta:(0,\infty)\times\Sigma_m\longrightarrow\Omega_m,
\qquad \Theta(\rho,y)=\rho y,
\]
is a diffeomorphism.  Indeed, its inverse is
\[
x\longmapsto\left(\ell_m(x),\frac{x}{\ell_m(x)}\right),
\]
using the linearity of $\ell_m$ and the fact that $\Omega_m$ is a cone.

The induced measure on the slice can be written explicitly.
Choose local coordinates $z=(z_1,\ldots,z_{N-1})$ on $\Sigma_m$ and write
$y=y(z)$.  Then
\[
\frac{\partial\Theta}{\partial\rho}=y,
\qquad
\frac{\partial\Theta}{\partial z_i}
=\rho\frac{\partial y}{\partial z_i}.
\]
Consequently the Jacobian is
\[
\rho^{N-1}
\left|\det\left(y,\frac{\partial y}{\partial z_1},\ldots,
\frac{\partial y}{\partial z_{N-1}}\right)\right|.
\]
Define the smooth positive measure on the slice by
\begin{equation}\label{eq:nu-definition}
\dd\nu(y)=
\left|\det\left(y,\frac{\partial y}{\partial z_1},\ldots,
\frac{\partial y}{\partial z_{N-1}}\right)\right|
\dd z_1\cdots\dd z_{N-1}.
\end{equation}
The ordinary transformation law for densities shows that this definition is
independent of the chosen coordinates.  It gives the measure identity
\begin{equation}\label{eq:cone-measure}
\dd x=\rho^{N-1}\dd\rho\dd\nu(y).
\end{equation}
The exponent $N-1$ occurs because the radial column is unchanged, whereas
each of the $N-1$ tangent columns is multiplied by $\rho$.

Equivalently, if $\dd\sigma$ is surface measure on the affine hyperplane
$(\ell_m=1)$, then
\begin{equation}\label{eq:nu-surface}
\dd\nu=\frac{\dd\sigma}{\|\nabla\ell_m\|}.
\end{equation}
Indeed, for the unit normal
$n=\nabla\ell_m/\|\nabla\ell_m\|$ and $y\in\Sigma_m$, one has
$y\cdot n=\ell_m(y)/\|\nabla\ell_m\|=1/\|\nabla\ell_m\|$.
In the present coordinates,
\[
\|\nabla\ell_m\|^2
=\sum_{j=0}^{r}\binom rj^2=\binom{2r}{r}.
\]

Recall that the Gamma function is initially defined by Euler's integral
\begin{equation}\label{eq:gamma-definition}
\Gamma(z)=\int_0^\infty \e^{-\rho}\rho^{z-1}\dd\rho,
\qquad \operatorname{Re}z>0.
\end{equation}
It extends meromorphically to $\C$, has no zeros, and has simple poles exactly
at $0,-1,-2,\ldots$; moreover,
$\operatorname{Res}_{z=-n}\Gamma(z)=(-1)^n/n!$ for $n\geq0$.  These standard
facts, as well as the identity $\Gamma(z+1)=z\Gamma(z)$ used below, are
recalled in \cite[Section 5.2]{DLMF}.

Using \eqref{eq:cone-measure} and \eqref{eq:gamma-definition}, we obtain
\begin{equation}\label{eq:polar}
\cZ_m(s)
=\Gamma(ms+N)\cA_m(s),
\qquad
\cA_m(s)=\int_{\Sigma_m}f_m(y)^s\dd\nu(y).
\end{equation}
Indeed, $\ell_m(\rho y)=\rho$ and
$f_m(\rho y)^s=\rho^{ms}f_m(y)^s$, so the radial factor is
\[
\int_0^\infty\e^{-\rho}\rho^{ms+N-1}\dd\rho=\Gamma(ms+N).
\]
The values $ms+N=0,-1,-2,\ldots$ are poles of the radial Gamma factor, but
they may be cancelled by $\cA_m(s)$.  We compute the latter next.

\subsection{Reduction to the Morris integral}

The moment coordinates separate the positive weights from the real nodes.
After integrating the weights, the remaining Vandermonde integral is carried
to the unit circle by $u=\tan(\theta/2)$ and evaluated by the Morris formula.

Set
\begin{equation}\label{eq:tau}
\tau=s+2.
\end{equation}
The positive definiteness of $H_m(x)$ is used in
Theorem~\ref{thm:radau} to make the operator $A$ self-adjoint.  Its eigenvalues, and
hence the nodes $u_1,\ldots,u_r$, are therefore real.  The weights are squared
orthogonal projections, so they are nonnegative; the cyclicity argument in the
proof of Theorem~\ref{thm:radau} shows that none of them vanishes, and hence they are
strictly positive.  Conversely, the argument following
Corollary~\ref{cor:factorization} shows that arbitrary distinct real nodes
$u_0=0,u_1,\ldots,u_r$ and positive weights give a point of $\Omega_m$,
and establishes the smooth coordinate change on the full-measure subset.
Thus we integrate $u_1,\ldots,u_r$ over $\R^r$ and
$w_0,\ldots,w_r$ over $(0,\infty)^{r+1}$.  The excluded conditions $u_a=0$ and $u_a=u_b$ define
measure-zero subsets of $\R^r$, so they do not change the value of the
integral.  We use the abbreviations
\[
 \dd w=\dd w_0\cdots\dd w_r,
 \qquad
 \dd u=\dd u_1\cdots\dd u_r.
\]
The representation in Theorem~\ref{thm:radau} is unique only as an unordered set of
pairs $(u_a,w_a)$, $1\leq a\leq r$.  If the variables range over all of
$\R^r$, every generic $x$ therefore has the $r!$ preimages obtained by
simultaneously permuting the nonzero nodes and their weights.  The collision
loci and the exceptional algebraic locus have measure zero.  Hence the
finite-covering change-of-variables formula requires division by $r!$.
Equivalently, one could impose $u_1<\cdots<u_r$ and omit this factor.  Using
Theorem~\ref{thm:radau} and Corollary~\ref{cor:factorization}, we obtain
\begin{align}
\cZ_m(s)
={}&\frac1{r!}\int_{(0,\infty)^{r+1}\times\R^r}
\e^{-w_0-\sum_{a=1}^{r}w_a(1+u_a^2)^r}
w_0^s\prod_{a=1}^{r}w_a^{s+1}
\notag\\[-2mm]
&\hspace{20mm}\cdot
\prod_{a=1}^{r}|u_a|^{2s+2}|\Delta(u)|^{2s+4}
\dd w\dd u.
\label{eq:Z-coordinates}
\end{align}
For every $c>0$ and $\operatorname{Re}z>0$,
\[
\int_0^\infty\e^{-cw}w^{z-1}\dd w=c^{-z}\Gamma(z).
\]
Separating the $w$-integrations in \eqref{eq:Z-coordinates} gives
\begin{align}
\cZ_m(s)
={}&\frac1{r!}
\left(\int_0^\infty \e^{-w_0}w_0^{\tau-2}\dd w_0\right)
\int_{\R^r}
\prod_{a=1}^{r}|u_a|^{2\tau-2}|\Delta(u)|^{2\tau}
\notag\\[-1mm]
&\hspace{15mm}\cdot
\prod_{a=1}^{r}
\left(\int_0^\infty
 \e^{-(1+u_a^2)^r w_a}w_a^{\tau-1}\dd w_a\right)\dd u
\notag\\
={}&\frac{\Gamma(\tau-1)\Gamma(\tau)^r}{r!}
\int_{\R^r}
\prod_{a=1}^{r}|u_a|^{2\tau-2}(1+u_a^2)^{-r\tau}
|\Delta(u)|^{2\tau}\dd u
\notag\\
={}&\frac{\Gamma(\tau-1)\Gamma(\tau)^r}{r!}J_r(\tau),
\label{eq:Z-J}
\end{align}
where
\begin{equation}\label{eq:J-real}
J_r(\tau)=\int_{\R^r}
\prod_{a=1}^{r}|u_a|^{2\tau-2}(1+u_a^2)^{-r\tau}
|\Delta(u)|^{2\tau}\dd u.
\end{equation}

Put
\begin{equation}\label{eq:half-angle}
u_a=\tan\frac{\theta_a}{2},
\qquad z_a=\e^{\ii\theta_a},
\qquad -\pi<\theta_a<\pi.
\end{equation}
Then
\begin{align*}
|u_a|&=\frac{|1-z_a|}{|1+z_a|},&
1+u_a^2&=\frac4{|1+z_a|^2},\\
\dd u_a&=\frac2{|1+z_a|^2}\dd\theta_a,&
|u_a-u_b|&=\frac{2|z_a-z_b|}{|1+z_a||1+z_b|}.
\end{align*}
These formulas follow from the identities
\[
1-z_a=-2\ii\e^{\ii\theta_a/2}\sin\frac{\theta_a}{2},
\qquad
1+z_a=2\e^{\ii\theta_a/2}\cos\frac{\theta_a}{2},
\]
whose quotient has absolute value $|\tan(\theta_a/2)|=|u_a|$.  Moreover,
$|1+z_a|^2=4\cos^2(\theta_a/2)$ and
$1+\tan^2 t=\sec^2t$, which gives the formula for $1+u_a^2$.
Differentiation gives
\[
\dd u_a=\frac12\sec^2\frac{\theta_a}{2}\dd\theta_a
=\frac2{|1+z_a|^2}\dd\theta_a.
\]
Finally,
\[
u_a=\ii\frac{1-z_a}{1+z_a}
\quad\Longrightarrow\quad
u_a-u_b=\frac{2\ii(z_b-z_a)}{(1+z_a)(1+z_b)},
\]
which proves the difference formula after taking absolute values.

Substitute these identities into the integrand in \eqref{eq:J-real}, and put
$\dd\theta=\dd\theta_1\cdots\dd\theta_r$.  The four groups of factors become
\begin{align*}
\prod_{a=1}^{r}|u_a|^{2\tau-2}
&=\prod_{a=1}^{r}|1-z_a|^{2\tau-2}|1+z_a|^{-2\tau+2},\\
\prod_{a=1}^{r}(1+u_a^2)^{-r\tau}
&=\prod_{a=1}^{r}
  \left(\frac{4}{|1+z_a|^2}\right)^{-r\tau}
 =2^{-2r^2\tau}\prod_{a=1}^{r}|1+z_a|^{2r\tau},\\
\dd u
&=\prod_{a=1}^{r}\left(\frac{2}{|1+z_a|^2}\,\dd\theta_a\right)
 =2^r\prod_{a=1}^{r}|1+z_a|^{-2}\dd\theta,\\
|\Delta(u)|^{2\tau}
&=\prod_{a<b}
  \left(\frac{2|z_a-z_b|}{|1+z_a||1+z_b|}\right)^{2\tau}\\
&=2^{2\tau\binom r2}
  \prod_{a<b}|z_a-z_b|^{2\tau}
  \prod_{a=1}^{r}|1+z_a|^{-2\tau(r-1)}.
\end{align*}
The last equality uses the fact that a fixed index $a$ occurs in exactly
$r-1$ pairs among the pairs $(a,b)$ with $a<b$; equivalently,
\[
 \prod_{a<b}(|1+z_a||1+z_b|)^{-2\tau}
 =\prod_{a=1}^{r}|1+z_a|^{-2\tau(r-1)}.
\]
Multiplying the four displayed identities gives the transformed density
\begin{align*}
&\prod_{a=1}^{r}|u_a|^{2\tau-2}(1+u_a^2)^{-r\tau}
 |\Delta(u)|^{2\tau}\dd u\\
&\quad=2^{r-2r^2\tau+2\tau\binom r2}
 \prod_{a=1}^{r}|1-z_a|^{2\tau-2}
 \prod_{a=1}^{r}|1+z_a|^{E_a}
 \prod_{a<b}|z_a-z_b|^{2\tau}\dd\theta,
\end{align*}
where the exponent attached to a fixed $|1+z_a|$ is
\begin{align*}
E_a
&=(-2\tau+2)+2r\tau-2-2\tau(r-1)\\
&=-2\tau+2+2r\tau-2-2r\tau+2\tau\\
&=0.
\end{align*}
Here the four summands in the first line come, in order, from
$|u_a|^{2\tau-2}$, $(1+u_a^2)^{-r\tau}$, the differential
$\dd u_a$, and the $r-1$ pairwise-difference factors containing $u_a$.
\Needspace{14\baselineskip}
Thus every factor $|1+z_a|$ cancels.  The remaining power of $2$ simplifies
as follows:
\begin{align*}
r-2r^2\tau+2\tau\binom r2
&=r-2r^2\tau+2\tau\frac{r(r-1)}2\\
&=r-2r^2\tau+r(r-1)\tau\\
&=r-2r^2\tau+(r^2-r)\tau\\
&=r-r^2\tau-r\tau\\
&=r-r(r+1)\tau\\
&=r-rm\tau,
\end{align*}
because $m=r+1$.  Consequently
\begin{equation}\label{eq:J-circle}
J_r(\tau)=2^{r-rm\tau}
\int_{(-\pi,\pi)^r}
\prod_{a=1}^{r}|1-z_a|^{2\tau-2}
\prod_{a<b}|z_a-z_b|^{2\tau}\dd\theta.
\end{equation}

We apply the Morris integral in the normalization of
\cite[(1.17)--(1.18)]{ForresterWarnaar}:
\begin{align}
&\frac1{(2\pi)^r}
\int_{(-\pi,\pi)^r}
\prod_{a=1}^{r}
\e^{\frac{\ii}{2}(A-B)\theta_a}|1+z_a|^{A+B}
\prod_{a<b}|z_a-z_b|^{2\gamma}\dd\theta
\notag\\
&\qquad=
\prod_{j=0}^{r-1}
\frac{\Gamma(1+A+B+j\gamma)\Gamma(1+(j+1)\gamma)}
{\Gamma(1+A+j\gamma)\Gamma(1+B+j\gamma)\Gamma(1+\gamma)}.
\label{eq:morris}
\end{align}
Its initial convergence conditions are
\[
\operatorname{Re}(A+B+1)>0,
\qquad
\operatorname{Re}\gamma>-\min\left\{\frac1r,
\frac{\operatorname{Re}(A+B+1)}{r-1}\right\},
\]
when $r>1$; for $r=1$, only the first condition is needed.  We take
\begin{equation}\label{eq:morris-sub}
A=B=\tau-1,
\qquad \gamma=\tau.
\end{equation}
For $\operatorname{Re}s>0$, we have $\operatorname{Re}\tau>2$, so all these
conditions hold.  The phase
in \eqref{eq:morris} disappears, and a rotation $z_a\mapsto-z_a$ replaces
$|1+z_a|$ by $|1-z_a|$.  Thus
\begin{align}
&\int_{(-\pi,\pi)^r}
\prod_a|1-z_a|^{2\tau-2}
\prod_{a<b}|z_a-z_b|^{2\tau}\dd\theta
\notag\\
&\quad=(2\pi)^r\prod_{j=0}^{r-1}
\frac{\Gamma((j+2)\tau-1)\Gamma(1+(j+1)\tau)}
{\Gamma((j+1)\tau)^2\Gamma(1+\tau)}.
\label{eq:morris-specialized}
\end{align}
\begin{samepage}
Using $\Gamma(1+z)=z\Gamma(z)$, we compute
\begin{align}
\prod_{j=0}^{r-1}
\frac{\Gamma(1+(j+1)\tau)}
{\Gamma((j+1)\tau)^2\Gamma(1+\tau)}
={}&\prod_{j=0}^{r-1}
\frac{(j+1)\tau\,\Gamma((j+1)\tau)}
{\Gamma((j+1)\tau)^2\,\tau\Gamma(\tau)}
\notag\\
={}&\prod_{j=0}^{r-1}
\frac{j+1}{\Gamma((j+1)\tau)\Gamma(\tau)}
\notag\\
={}&\frac{r!}{\Gamma(\tau)^r\prod_{j=1}^{r}\Gamma(j\tau)}.
\label{eq:morris-cancel}
\end{align}
\end{samepage}
Substitution into \eqref{eq:Z-J} proves the following theorem.

\begin{theorem}
\label{thm:zeta-formula}
For $m\geq1$, with $r=m-1$ and $\tau=s+2$, the meromorphic continuation of
\eqref{eq:Z-def} is
\begin{equation}\label{eq:zeta-formula}
\cZ_m(s)=
2^{r-rm\tau}(2\pi)^r\Gamma(\tau-1)
\prod_{j=2}^{m}
\frac{\Gamma(j\tau-1)}{\Gamma((j-1)\tau)}.
\end{equation}
For $m=1$, the formula is interpreted as
$\cZ_1(s)=\Gamma(s+1)$.
\end{theorem}

\begin{proof}
For $m\geq2$, equations
\eqref{eq:Z-coordinates}--\eqref{eq:morris-cancel} prove
\eqref{eq:zeta-formula} in the nonempty open half-plane
$\operatorname{Re}s>0$.  The right-hand side is meromorphic on $\C$, so it is
the meromorphic continuation of the integral.  If $m=1$, then
$\Omega_1=(0,\infty)$ and
$f_1=\ell_1=x_0$, which gives the stated Gamma integral directly.
\end{proof}

\subsection{The primitive poles and angular noncancellation}

For $m\geq2$, define the set of proposed positive root exponents of exact
denominator $m$ by
\begin{equation}\label{eq:primitive-exponents}
\cR_m=
\left\{2-\frac1m\right\}
\cup
\left\{2+\frac{k}{m}:1\leq k\leq m-2,\ (k,m)=1\right\}.
\end{equation}
Its fractional parts run through all primitive residue classes modulo $m$, and
$|\cR_m|=\varphi(m)$.

The last numerator in \eqref{eq:zeta-formula} is
\begin{equation}\label{eq:radial-last}
\Gamma(m\tau-1)=\Gamma(ms+N),
\end{equation}
which agrees with the radial factor in \eqref{eq:polar}.  Dividing it out gives
the meromorphic continuation of the angular factor:
\begin{equation}\label{eq:A-formula}
\cA_m(s)=
2^{r-rm\tau}(2\pi)^r\Gamma(\tau-1)
\left[\prod_{j=2}^{m-1}
\frac{\Gamma(j\tau-1)}{\Gamma((j-1)\tau)}\right]
\frac1{\Gamma((m-1)\tau)}.
\end{equation}

\begin{prop}
\label{prop:no-cancellation}
For every $\alpha\in\cR_m$, the function $\cA_m(s)$ is finite and
nonzero at $s=-\alpha$, while $\Gamma(ms+N)$ has a simple pole there.
Consequently $\cZ_m(s)$ has a genuine simple pole at every
$s=-\alpha$, $\alpha\in\cR_m$.
\end{prop}

\begin{proof}
If $\alpha=2-1/m$, then $\tau=1/m$ and $m\tau-1=0$, so the radial factor has
a simple pole.  Every Gamma argument in \eqref{eq:A-formula} is congruent
modulo $\Z$ to $a/m$ for some $1\leq a\leq m-1$, and hence is nonintegral.
The angular factor is therefore finite and nonzero.

If $\alpha=2+k/m$, where $1\leq k\leq m-2$ and $(k,m)=1$, then
$\tau=-k/m$ and $m\tau-1=-k-1$, again a radial pole.  Any pole or zero of the
Gamma quotient in \eqref{eq:A-formula} would require one of
$\tau-1$, $j\tau-1$, $(j-1)\tau$ with $2\leq j<m$, or
$(m-1)\tau$ to be an integer.  Since $(k,m)=1$, this would force $m$ to divide
one of $1,j,j-1,m-1$, which is impossible in the displayed ranges.  Thus the
angular factor is again finite and nonzero.  Since Gamma has no zeros and all
of its poles are simple, the asserted poles of $\cZ_m$ are genuine and
simple.
\end{proof}

Here $\Gamma(ms+N)$ is the contribution of the vertex divisor, while
\eqref{eq:A-formula} accounts for the directions on that divisor.
Proposition~\ref{prop:no-cancellation} shows that the primitive poles survive.

\section{From zeta poles to roots at exact positions}\label{sec:poles-roots}

\subsection{The pole--root implication and its integer shift}

An Archimedean zeta pole determines a Bernstein--Sato root only up to a
nonnegative integral shift.  The branchwise statement below, together with
\eqref{eq:N-over-m}, will determine the translate.

The proposition is stated for a single connected component of $(f>0)$.  We
write $C_c^\infty(\R^N)$ for the space of smooth, compactly supported test
functions on $\R^N$.

\begin{prop}
\label{prop:pole-root}
Let $f\in\R[x_1,\ldots,x_N]$ be nonzero, let $\Omega$ be a connected
component of $(f>0)$, and let $\phi\in C_c^\infty(\R^N)$.  The integral
\[
Z_{\Omega,\phi}(s)=\int_\Omega\phi(x)f(x)^s\dd x,
\qquad \operatorname{Re}s>0,
\]
has a meromorphic continuation.  If $s_0$ is a pole, then
\begin{equation}\label{eq:pole-root}
b_f(s_0+j)=0
\qquad\text{for some }j\in\Z_{\geq0}.
\end{equation}
The same conclusion holds with the local Bernstein--Sato polynomial when the
support of $\phi$ is contained in the corresponding coordinate neighborhood.
\end{prop}

\begin{proof}
Because $\Omega$ is a connected component of $(f>0)$, one has
$\partial\Omega\subset(f=0)$.  Bernstein's region-distribution construction
applies to the function equal to $f^s$ on $\Omega$ and to zero outside
$\Omega$.  Theorem~1 gives its meromorphic continuation, while Lemma~2.1
states that algebraic differential identities for formal powers of $f$
pass to this extended family of distributions
\cite[Theorem 1 and Lemma 2.1]{Bernstein}.  This statement is componentwise:
the chosen region is treated separately, rather than only after summing over
all components of $\R^N\setminus(f=0)$.

Fix an algebraic Bernstein identity
$P(s)f^{s+1}=b_f(s)f^s$ and write
\[
 P(s)=\sum_{|\gamma|\leq d}a_\gamma(x,s)\partial_x^\gamma,
 \qquad
 P(s)^t=\sum_{|\gamma|\leq d}(-\partial_x)^\gamma\circ
 a_\gamma(x,s).
\]
Here $P(s)^t$ is the formal transpose.  For a test function
$\phi\in C_c^\infty(\R^N)$, the region-distribution theorem gives, initially
for $\operatorname{Re}s\gg0$,
\begin{equation}\label{eq:distributional-bernstein-pairing}
 b_f(s)\int_\Omega f^s\phi\,\dd x
 =\int_\Omega P(s)(f^{s+1})\phi\,\dd x
 =\int_\Omega f^{s+1}P(s)^t\phi\,\dd x.
\end{equation}
The first equality uses the algebraic Bernstein identity on $\Omega$; only
the second involves integration by parts.  For each multi-index $\gamma$, the
second equality is obtained by applying one-variable integration by parts
successively in the coordinate directions specified by
$\partial_x^\gamma$.  This is the usual local calculation at a regular
boundary point: since $f=0$ there and $\operatorname{Re}s$ is initially
sufficiently large, all boundary terms arising from the finitely many
integrations vanish.  At a singular boundary point, Bernstein's theorem
establishes the same identity directly for the extension by zero, without
invoking a smooth-boundary formula.  The inclusion
$\partial\Omega\subset(f=0)$ is essential here: it ensures that no artificial
boundary lies inside $(f\neq0)$, while the compact support of $\phi$ removes
boundary terms at infinity.  The theorem therefore treats each component and
the singular part of its boundary without requiring a resolution.

Both sides of \eqref{eq:distributional-bernstein-pairing} are meromorphic
families supplied by Bernstein's theorem.  Since they agree in a sufficiently
far right half-plane, the identity theorem extends the equality to all $s$.

Apply the integral identity successively at $s,s+1,\ldots,s+K-1$, replacing
the test function at each step by its image under the preceding formal
transposes.  This gives
\begin{equation}\label{eq:iterated-Bernstein}
\left(\prod_{j=0}^{K-1}b_f(s+j)\right)
 Z_{\Omega,\phi}(s)
=\int_\Omega f^{s+K}
 P(s+K-1)^t\cdots P(s)^t\phi\,\dd x.
\end{equation}
For a fixed $s_0$, choose $K$ so large that $\operatorname{Re}(s+K)>0$ near
$s_0$.  The last integral is then holomorphic there: the transformed test
function remains smooth and compactly supported, with coefficients depending
polynomially on $s$.

If every $b_f(s_0+j)$, $0\leq j<K$, were nonzero, the scalar factor on the
left of \eqref{eq:iterated-Bernstein} would be invertible near $s_0$.
Consequently $Z_{\Omega,\phi}$ would be holomorphic there, contradicting the
assumption that $s_0$ is a pole.  This proves \eqref{eq:pole-root}.  Using a
local Bernstein equation throughout proves the final assertion.
\end{proof}

Only the implication in Proposition~\ref{prop:pole-root} is used here.  Its
converse is false for a fixed branch and test function: the Bernstein
polynomial provides a universal algebraic denominator, whereas residues may
vanish after restriction to a real branch or pairing with a particular test
function.  A discussion of the same inclusion for Archimedean local zeta
functions can be found in
\cite[Section 2.2]{Denef} and \cite[Section 5.3]{Igusa}.

For comparison, a whole-space zeta integral decomposes, before meromorphic
continuation, into its sign components:
\[
\int_{\R^N}\phi|f|^s\dd x
=\sum_{\Omega_i\subset(f>0)}Z_{\Omega_i,\phi}(s)
+\sum_{\Omega_j^-\subset(f<0)}
\int_{\Omega_j^-}\phi(-f)^s\dd x.
\]
This identity does not imply that the pole set of the sum contains the pole
set of every summand, since principal parts can cancel.
Proposition~\ref{prop:pole-root} gives the stronger, logically independent
statement that the same Bernstein--Sato polynomial controls each component
distribution separately.  In our application $\Omega_m$ is one such
component, and its boundary consists of singular positive semidefinite Hankel
matrices, on which $f_m=0$.

\subsection{Localization of the cone integral at the vertex}

The exponential in \eqref{eq:Z-def} is not compactly supported, so
Proposition~\ref{prop:pole-root} does not apply directly.  Each
primitive pole is detected by a compactly supported test function in an
arbitrarily small neighborhood of the origin.  The construction preserves the
polar-coordinate factorization and leaves the relevant radial principal parts
unchanged.

\begin{prop}
\label{prop:localized-poles}
Let $W\subset\R^N$ be any open neighborhood of the origin.  For every
$\alpha\in\cR_m$, there exists
$\phi\in C_c^\infty(W)$ such that the branch zeta integral
\[
 \int_{\Omega_m}\phi(x)f_m(x)^s\dd x
\]
has a genuine simple pole at $s=-\alpha$.
\end{prop}

\begin{proof}
Let $W\subset\R^N$ be an arbitrary open neighborhood of $0$.  Choose
$\delta>0$ such that the closed ball $\{\|x\|\leq\delta\}$ is contained in
$W$, and then choose
\begin{equation}\label{eq:R-smallness}
 0<R<\frac{3\delta}{2C_m},
\end{equation}
where $C_m$ is the constant in \eqref{eq:ell-norm}.  Take
$\chi_R\in C_c^\infty(\R)$ with $0\leq\chi_R\leq1$ and
\[
\chi_R(\rho)=1\quad(|\rho|\leq R/3),
\qquad
\chi_R(\rho)=0\quad(|\rho|\geq2R/3).
\]
By Lemma~\ref{lem:ell-coercive},
\[
K_R=\wb\Omega_m\cap\{0\leq\ell_m\leq2R/3\}
\]
is compact, and \eqref{eq:ell-norm} and \eqref{eq:R-smallness} give
\[
 K_R\subseteq\{\|x\|\leq2C_mR/3\}
 \subseteq\{\|x\|<\delta\}\subseteq W.
\]
The smooth Urysohn lemma therefore supplies
$\eta_R\in C_c^\infty(\R^N)$ with $\operatorname{supp}\eta_R\subseteq W$
which is equal to $1$ on an open neighborhood of $K_R$.  Define
\begin{equation}\label{eq:phi-R}
\phi_R(x)=\eta_R(x)\chi_R(\ell_m(x))\e^{-\ell_m(x)}.
\end{equation}
This is a smooth compactly supported function with
$\operatorname{supp}\phi_R\subseteq W$.  Moreover, if $x\in\Omega_m$ and
$\chi_R(\ell_m(x))\neq0$, then $x\in K_R$, so $\eta_R(x)=1$.  The auxiliary
factor $\eta_R$ therefore makes the test function compactly supported without
changing its values on the part of the cone retained by the radial cutoff.

In the polar coordinates of
\eqref{eq:polar}, write $x=\rho y$, where
$\rho=\ell_m(x)>0$ and $y\in\Sigma_m$.  Then
\[
 f_m(\rho y)=\rho^m f_m(y),
 \qquad
 \dd x=\rho^{N-1}\dd\rho\dd\nu(y),
 \qquad
 \ell_m(\rho y)=\rho.
\]

For $\operatorname{Re}s>0$, absolute convergence and Fubini's theorem now give
\begin{align}\label{eq:localized-zeta}
\cZ_{m,R}(s)
&:=\int_{\Omega_m}\phi_R(x)f_m(x)^s\dd x \notag\\
&=\int_{\Sigma_m}\int_0^\infty
 \chi_R(\rho)\e^{-\rho}
 \rho^{ms+N-1} f_m(y)^s\dd\rho\dd\nu(y) \notag\\
&=\left(\int_0^\infty
 \chi_R(\rho)\e^{-\rho}\rho^{ms+N-1}\dd\rho\right)
 \left(\int_{\Sigma_m}f_m(y)^s\dd\nu(y)\right) \notag\\
&=\Gamma_R(ms+N)\cA_m(s),
\end{align}
where
\begin{equation}\label{eq:cutoff-gamma}
\Gamma_R(z)=\int_0^\infty
\chi_R(\rho)\e^{-\rho}\rho^{z-1}\dd\rho.
\end{equation}
The second line uses the construction of $\eta_R$, which gives
$\eta_R(\rho y)=1$ whenever
$\chi_R(\rho)\neq0$.  The power $\rho^{ms+N-1}$ is the product of
$\rho^{ms}$ from the homogeneity of $f_m^s$ and $\rho^{N-1}$ from the
Jacobian.  The remaining factors separate into a function of $\rho$ and a
function of $y$, so Fubini gives the product in the third line.  Its two
factors are $\Gamma_R(ms+N)$ and the angular integral
$\cA_m(s)$ from \eqref{eq:polar}.  This identity is initially valid for
$\operatorname{Re}s>0$; meromorphic continuation follows from the comparison
below.

For $\operatorname{Re}z>0$,
\begin{equation}\label{eq:cutoff-gamma-entire}
\Gamma_R(z)-\Gamma(z)
=-\int_0^\infty
(1-\chi_R(\rho))\e^{-\rho}\rho^{z-1}\dd\rho
\end{equation}
and the right-hand side extends to an entire function of $z$.  Indeed,
$1-\chi_R$ vanishes for
$0<\rho<R/3$, so there is no singularity at $\rho=0$.  Fix a compact set
$K\subset\C$ and a derivative order $k\geq0$, and choose $M>0$ such that
$|\operatorname{Re}z-1|\leq M$ on $K$.  On $[R/3,\infty)$, the absolute
value of the $k$-th $z$-derivative of the integrand is bounded uniformly for
$z\in K$ by an integrable function of the form
\[
 C_K\e^{-\rho}\bigl(\rho^M+\rho^{-M}\bigr)
 \bigl(1+|\log\rho|^k\bigr).
\]
Differentiation under the integral sign is therefore valid to every order, so
the difference in \eqref{eq:cutoff-gamma-entire} is entire.  Consequently
$\Gamma_R$ extends meromorphically and has the same Laurent principal part as
$\Gamma$ at every pole of $\Gamma$.  Together with the continuation of
$\cA_m$ in \eqref{eq:A-formula}, this gives the meromorphic continuation of
\eqref{eq:localized-zeta}.

For $\alpha\in\cR_m$, the number $N-m\alpha$ is $0$ or $-k-1$, hence a
nonpositive integer.  By \eqref{eq:cutoff-gamma-entire},
$\Gamma_R(ms+N)$ therefore has a simple pole at $s=-\alpha$ with the same nonzero
principal coefficient as $\Gamma(ms+N)$.  By
Proposition~\ref{prop:no-cancellation}, $\cA_m$ is holomorphic and nonzero there,
so their product in \eqref{eq:localized-zeta} has a genuine simple pole.
Because $W$ was arbitrary, every neighborhood of the vertex admits a test
function supported in it that detects this pole; equivalently, $-\alpha$ is a
pole of the local branch zeta distribution at $0$.
\end{proof}

We also need the following standard fact for homogeneous polynomials.

\begin{lemma}
\label{lem:homogeneous-localization}
If $f$ is a nonzero homogeneous polynomial on a complex vector space, then
its global Bernstein--Sato polynomial equals its local polynomial at the
origin:
\[
 b_f(s)=b_{f,0}(s).
\]
\end{lemma}

\begin{proof}
Choose a neighborhood $U_0$ of the origin on which a Bernstein equation for
$b_{f,0}$ is defined.  For every point $x$, a nonzero scalar dilation sends
$x$ into $U_0$ and changes $f$ only by a nonzero constant.  Local
Bernstein--Sato polynomials are invariant under such an analytic coordinate
change and constant rescaling, so $b_{f,x}$ divides $b_{f,0}$.  Conversely,
the global polynomial is the least common multiple of all local polynomials
on the affine space, and in particular is divisible by $b_{f,0}$; see
\cite[Section 2.2]{Denef}.  Hence equality holds.
\end{proof}

\subsection{The new roots}

\begin{prop}\label{prop:primitive-roots}
For every $\alpha\in\cR_m$,
\[
b_{f_m}(-\alpha)=0.
\]
Each of these roots has multiplicity one.
\end{prop}

\begin{proof}
Choose the compactly supported test function supplied by
Proposition~\ref{prop:localized-poles} and apply the local form of
Proposition~\ref{prop:pole-root}.  For some $j\geq0$,
\begin{equation}\label{eq:shifted-root}
b_{f_m,0}(-\alpha+j)=0.
\end{equation}
Write $\beta=\alpha-j$.  Then $-\beta$ is a local root and, by
Theorem~\ref{thm:roots-give-monodromy}, with the sign convention fixed by
\eqref{eq:wu-rh-comparison}, its corresponding monodromy eigenvalue is
\[
\exp(2\pi\ii\beta)=\exp(2\pi\ii\alpha),
\]
which is primitive of order $m$.  The reduced denominator of $\beta$ is therefore
$m$.  In the resolution candidate formula
\eqref{eq:resolution-candidates}, no index $q<m$ can produce a fraction of
reduced denominator $m$.  Thus the relevant index is $q=m$, and hence
\begin{equation}\label{eq:beta-lower}
\beta\geq\frac Nm=2-\frac1m.
\end{equation}

If $\alpha=2-1/m$, then $j\geq1$ contradicts
\eqref{eq:beta-lower}.  If $\alpha=2+k/m$ with
$1\leq k\leq m-2$, then $j\geq1$ gives
\[
\beta\leq1+\frac{k}{m}
\leq2-\frac2m
<2-\frac1m,
\]
again a contradiction.  Thus $j=0$ in \eqref{eq:shifted-root}.

By Lemma~\ref{lem:homogeneous-localization}, the local root is a root of the
global polynomial.  Multiplicity one follows from
Corollary~\ref{cor:one-root-class}.
\end{proof}

\section{Induction and proof of the main theorem}\label{sec:induction}

\subsection{Inheritance through transverse slices}

\begin{lemma}\label{lem:divisibility}
If $1\leq q<m$, then
\begin{equation}\label{eq:divisibility}
b_{f_q}(s)\mid b_{f_m}(s).
\end{equation}
\end{lemma}

\begin{proof}
It is enough to prove
\[
b_{f_{k-1}}(s)\mid b_{f_k}(s)
\]
for every $2\leq k\leq m$.  Consider the secant stratification associated
with $f_k$, and choose a general point $x\in X_1(2k-2)\setminus\{0\}$.
Applying Theorem~\ref{thm:brogan}\ref{item:brogan-slice} with $m=k$ and
$q=k-1$, we find that the analytic germ of $f_k$ at $x$ is, up to an
analytic coordinate change and multiplication by a unit, equal to
$f_{k-1}$ in the transverse variables and is independent of the remaining
smooth variables.  These operations do not change the local
Bernstein--Sato polynomial.  Indeed, analytic coordinate changes transport
Bernstein operators.  If a germ is multiplied by a unit $u$, the graph
embeddings are identified by the ambient change $(y,t)\mapsto(y,u(y)t)$;
this change preserves the Euler operator $t\partial_t$ that defines the
$b$-function.  Finally, adjoining smooth variables gives an external product,
and restriction to a transverse smooth slice gives the inverse operation.
Hence
\[
b_{f_k,x}(s)=b_{f_{k-1},0}(s).
\]
Since $f_{k-1}$ is homogeneous,
Lemma~\ref{lem:homogeneous-localization} gives
\[
b_{f_{k-1},0}(s)=b_{f_{k-1}}(s).
\]
The global Bernstein--Sato polynomial is the least common multiple of the
local Bernstein--Sato polynomials.  Therefore
\[
b_{f_{k-1}}(s)=b_{f_k,x}(s)\mid b_{f_k}(s).
\]
Iterating these adjacent divisibilities gives
\[
b_{f_q}(s)\mid b_{f_m}(s)
\qquad (1\leq q<m).
\]
\end{proof}

Lemma~\ref{lem:divisibility} carries all roots from smaller sizes to $f_m$;
Proposition~\ref{prop:primitive-roots} supplies the new primitive order-$m$
roots at the vertex.

\subsection{Completion of the proof}

\begin{proof}[Proof of Theorem~\ref{thm:main-intro}]
We argue by induction on $m$.  For $m=1$, $f_1=x_0$, so
$b_{f_1}(s)=s+1$.

Assume the formula has been proved for every $q<m$.  We classify roots by the
order of their monodromy eigenvalue.

For the eigenvalue $1$, the universal factor $s+1$ occurs for every nonconstant
polynomial.  By Corollary~\ref{cor:one-root-class}, this is the only root in
the integral congruence class, and it is simple.

Let $2\leq q<m$ and let $\lambda$ be primitive of order $q$.  By the induction
hypothesis, $b_{f_q}$ has exactly one root in the congruence class of $\lambda$;
as $\lambda$ varies through the primitive $q$-th roots, these roots are
\[
-2+\frac1q,
\qquad
-2-\frac{k}{q}
\quad(1\leq k\leq q-2,\ (k,q)=1).
\]
They are roots of $b_{f_m}$ by Lemma~\ref{lem:divisibility}.  The uniqueness
and simplicity assertions in Corollary~\ref{cor:one-root-class} exclude every
other root in the same congruence classes and prevent any increase in
multiplicity.

For primitive eigenvalues of order $m$, Proposition~\ref{prop:primitive-roots} gives
exactly the roots indexed by $\cR_m$, and
Corollary~\ref{cor:one-root-class} again gives uniqueness and simplicity.

It remains to exclude additional congruence classes.  By
Theorem~\ref{thm:roots-give-monodromy}, every root $\xi$ of $b_{f_m}(s)$
determines the local monodromy eigenvalue $\exp(-2\pi\ii\xi)$.  The eigenvalue
list in Theorem~\ref{thm:brogan}\ref{item:brogan-eigen} contains only roots of
unity of orders $1,\ldots,m$, all of which have already been treated.  Hence
there are no further roots, and multiplying the simple factors gives
\eqref{eq:main-formula}.

Finally, every number in $\cR_q$ has reduced fractional denominator $q$.
Thus roots belonging to different $q$ are distinct, and the $q$-th group
contains $\varphi(q)$ roots.  This proves the degree formula as well.
\end{proof}

\section{Examples and computer checks}\label{sec:checks}

The examples in this section are consequences and checks of the proof; no
computer calculation is used above.

\subsection{The first four determinants}

The theorem gives
\begin{align*}
b_{f_1}(s)&=s+1,\\
b_{f_2}(s)&=(s+1)(s+\tfrac32),\\
b_{f_3}(s)&=(s+1)(s+\tfrac32)(s+\tfrac53)(s+\tfrac73),\\
b_{f_4}(s)&=(s+1)(s+\tfrac32)(s+\tfrac53)(s+\tfrac73)
(s+\tfrac74)(s+\tfrac94).
\end{align*}

For $m=3$, \eqref{eq:zeta-formula} becomes
\[
\cZ_3(s)=2^{2-6\tau}(2\pi)^2\Gamma(\tau-1)
\frac{\Gamma(2\tau-1)}{\Gamma(\tau)}
\frac{\Gamma(3\tau-1)}{\Gamma(2\tau)},
\qquad \tau=s+2.
\]
It has the two primitive poles $s=-5/3$ and $s=-7/3$.  It also has a pole at
$s=-8/3$, which is explained by the root $-5/3$ after the shift
$-8/3+1=-5/3$.  Thus the inclusion in
Proposition~\ref{prop:pole-root} need not be an equality of pole and root sets.

\subsection{Macaulay2}

The following commands use the \texttt{Dmodules} package in Macaulay2
\cite{Macaulay2}.  The computations for $m=3,4$ reproduce the displayed
factorizations.  These computations are independent checks and are not used in
the proof.

\begin{verbatim}
needsPackage "Dmodules"
D = QQ[x0,x1,x2,x3,x4,D0,D1,D2,D3,D4,
WeylAlgebra => {x0=>D0, x1=>D1, x2=>D2, x3=>D3, x4=>D4}]
H = matrix {
    {x0, x1, x2},
    {x1, x2, x3},
    {x2, x3, x4}
}
f = det H
b = globalBFunction(f)
factor b

needsPackage "Dmodules"
D = QQ[x0,x1,x2,x3,x4,x5,x6,D0,D1,D2,D3,D4,D5,D6,
WeylAlgebra => {x0=>D0, x1=>D1, x2=>D2, x3=>D3,
                x4=>D4, x5=>D5, x6=>D6}]
H = matrix {
    {x0, x1, x2, x3},
    {x1, x2, x3, x4},
    {x2, x3, x4, x5},
    {x3, x4, x5, x6}
}
f = det H
b = globalBFunction(f)
factor b
\end{verbatim}

\bibliographystyle{amsalpha}
\bibliography{references}

\providecommand{\bysame}{\leavevmode\hbox to3em{\hrulefill}\thinspace}
\providecommand{\MR}{\relax\ifhmode\unskip\space\fi MR }
\providecommand{\MRhref}[2]{%
  \href{https://mathscinet.ams.org/mathscinet-getitem?mr=#1}{#2}}
\providecommand{\href}[2]{#2}
\begin{thebibliography}{DLMF}

\bibitem[Ber72]{Bernstein}
I.~N. Bernstein,
\emph{The analytic continuation of generalized functions with respect to a
parameter}, Funct. Anal. Appl. \textbf{6} (1972), 273--285,
\href{https://doi.org/10.1007/BF01077645}{doi:10.1007/BF01077645}.

\bibitem[Ber92]{Bertram}
A. Bertram,
\emph{Moduli of rank-$2$ vector bundles, theta divisors, and the geometry of
curves in projective space}, J. Differential Geom. \textbf{35} (1992),
429--469,
\href{https://doi.org/10.4310/jdg/1214448083}
{doi:10.4310/jdg/1214448083}.

\bibitem[Bro26]{Brogan}
D. Brogan,
\emph{Invariants of the singularities of secant varieties of curves},
J.~\'{E}c. polytech. Math. \textbf{13} (2026), 1--39,
\href{https://doi.org/10.5802/jep.321}{doi:10.5802/jep.321}.

\bibitem[Den91]{Denef}
J. Denef,
\emph{Report on Igusa's local zeta function}, S\'{e}minaire Bourbaki,
Exp.~741, Ast\'{e}risque \textbf{201--203} (1991), 359--386,
\href{https://doi.org/10.24033/ast.122}{doi:10.24033/ast.122}.

\bibitem[FW08]{ForresterWarnaar}
P.~J. Forrester and S.~O. Warnaar,
\emph{The importance of the Selberg integral}, Bull. Amer. Math. Soc.
\textbf{45} (2008), 489--534,
\href{https://doi.org/10.1090/S0273-0979-08-01221-4}
{doi:10.1090/S0273-0979-08-01221-4}.

\bibitem[GS]{Macaulay2}
D.~R. Grayson and M.~E. Stillman,
\emph{Macaulay2, a software system for research in algebraic geometry},
available at \url{https://macaulay2.com/}.

\bibitem[HTT08]{HTT}
R. Hotta, K. Takeuchi, and T. Tanisaki,
\emph{$D$-modules, perverse sheaves, and representation theory},
Progress in Mathematics, vol.~236, Birkh\"auser Boston, 2008,
\href{https://doi.org/10.1007/978-0-8176-4523-6}
{doi:10.1007/978-0-8176-4523-6}.

\bibitem[Igu00]{Igusa}
J.-I. Igusa,
\emph{An Introduction to the Theory of Local Zeta Functions},
AMS/IP Studies in Advanced Mathematics, vol.~14, American Mathematical
Society and International Press, 2000,
\href{https://doi.org/10.1090/amsip/014}{doi:10.1090/amsip/014}.

\bibitem[Kas77]{Kashiwara}
M. Kashiwara,
\emph{$B$-functions and holonomic systems: rationality of roots of
$B$-functions}, Invent. Math. \textbf{38} (1976/77), 33--53,
\href{https://doi.org/10.1007/BF01390168}{doi:10.1007/BF01390168}.

\bibitem[Lic89]{Lichtin}
B. Lichtin,
\emph{Poles of $|f(z,w)|^{2s}$ and roots of the $b$-function},
Ark. Mat. \textbf{27} (1989), 283--304,
\href{https://doi.org/10.1007/BF02386377}{doi:10.1007/BF02386377}.

\bibitem[L\H{o}r20]{LorinczSlices}
A.~C. L\H{o}rincz,
\emph{Decompositions of Bernstein--Sato polynomials and slices},
Transform. Groups \textbf{25} (2020), no.~2, 577--607,
\href{https://doi.org/10.1007/s00031-019-09526-7}
{doi:10.1007/s00031-019-09526-7}.

\bibitem[LRWW17]{LRWW}
A.~C. L\H{o}rincz, C. Raicu, U. Walther, and J. Weyman,
\emph{Bernstein--Sato polynomials for maximal minors and sub-maximal
Pfaffians}, Adv. Math. \textbf{307} (2017), 224--252,
\href{https://doi.org/10.1016/j.aim.2016.11.011}
{doi:10.1016/j.aim.2016.11.011}.

\bibitem[Mal75]{Malgrange}
B. Malgrange,
\emph{Le polyn\^ome de Bernstein d'une singularit\'e isol\'ee}, in
\emph{Fourier Integral Operators and Partial Differential Equations}
(Colloq. Internat., Univ. Nice, Nice, 1974), Lecture Notes in Math., vol.~459,
Springer, Berlin, 1975, 98--119,
\href{https://doi.org/10.1007/BFb0074194}{doi:10.1007/BFb0074194}.

\bibitem[DLMF]{DLMF}
NIST Digital Library of Mathematical Functions,
\emph{Gamma Function}, Chapter~5,
\url{https://dlmf.nist.gov/5}.

\bibitem[Pop21]{PopaNotes}
M. Popa,
\emph{$D$-modules in birational geometry}, expanded lecture notes for
Math 296, Harvard University, Spring 2021,
\href{https://people.math.harvard.edu/~mpopa/notes/DMBG-posted.pdf}
{available online}.

\bibitem[Sat90]{SatoPVS}
M. Sato,
\emph{Theory of prehomogeneous vector spaces (algebraic part)---the English
translation of Sato's lecture from Shintani's note}, notes by T. Shintani,
translated from the Japanese by M. Muro, Nagoya Math. J. \textbf{120} (1990),
1--34,
\href{https://doi.org/10.1017/S0027763000003214}
{doi:10.1017/S0027763000003214}.

\bibitem[SS74]{SatoShintani}
M. Sato and T. Shintani,
\emph{On zeta functions associated with prehomogeneous vector spaces},
Ann. of Math. (2) \textbf{100} (1974), no.~1, 131--170,
\href{https://doi.org/10.2307/1970844}{doi:10.2307/1970844}.

\bibitem[SY26]{SchnellYang}
C. Schnell and R. Yang,
\emph{A log resolution for the theta divisor of a hyperelliptic curve},
Algebr. Geom. \textbf{13} (2026), no.~4, 462--503,
\href{https://doi.org/10.14231/AG-2026-014}{doi:10.14231/AG-2026-014}.

\bibitem[Wu21]{Wu}
L. Wu,
\emph{On the comparison of nearby cycles via $b$-functions},
J. Singul. \textbf{23} (2021), 92--106,
\href{https://doi.org/10.5427/jsing.2021.23e}
{doi:10.5427/jsing.2021.23e}.

\end{thebibliography}

\end{document}